\documentclass[11pt,final]{amsart}
\usepackage[utf8]{inputenc}
\usepackage[T1]{fontenc}
\usepackage[letterpaper,centering]{geometry}
\usepackage{lmodern}
\usepackage{mathrsfs}
\usepackage{amsmath,amssymb,amsfonts}
\usepackage[unicode,pdfborder={0 0 0}]{hyperref}
\usepackage{amscd}
\usepackage[ps,dvips,arrow,matrix,tips,line]{xy}
{\setbox0\hbox{$ $}}\fontdimen16\textfont2=\fontdimen17\textfont2
\entrymodifiers={+!!<0pt,\the\fontdimen22\textfont2>}
\SelectTips{cm}{11}

\theoremstyle{plain}
\newtheorem{thm}{Theorem}[section]
\newtheorem{thmintro}{Theorem}
\newtheorem{lem}[thm]{Lemma}

\newtheorem{prop}[thm]{Proposition}
\theoremstyle{definition}
\newtheorem{defn}[thm]{Definition}
\newtheorem{rmk}[thm]{Remark}
\newtheorem{rmks}[thm]{Remarks}
\newtheorem{example}[thm]{Example}
\newtheorem{examples}[thm]{Examples}
\numberwithin{equation}{section}

\newcommand{\Gm}{\mathbf{G}_\mathrm{m}}
\newcommand{\tors}[2]{{\vphantom{#2}}_{#1}{#2}}
\newcommand{\isoto}{\myxrightarrow{\,\sim\,}}

\makeatletter
\def\myrightarrow{{\setbox\z@\hbox{$\rightarrow$}\dimen0\ht\z@\multiply\dimen0 6\divide\dimen0 10\ht\z@\dimen0\box\z@}}
\def\myrightarrowfill@{\arrowfill@\relbar\relbar\myrightarrow}
\newcommand{\myxrightarrow}[2][]{\ext@arrow 0359\myrightarrowfill@{#1}{#2}}
\def\myleftarrow{{\setbox\z@\hbox{$\leftarrow$}\dimen0\ht\z@\multiply\dimen0 6\divide\dimen0 10\ht\z@\dimen0\box\z@}}
\def\myleftarrowfill@{\arrowfill@\myleftarrow\relbar\relbar}
\newcommand{\myxleftarrow}[2][]{\ext@arrow 3095\myleftarrowfill@{#1}{#2}}
\makeatother
\newcommand{\sXtilde}{\mkern4.5mu\widetilde{\mkern-4.5mu\sX}}
\newcommand{\sC}{{\mathscr C}}
\newcommand{\sL}{{\mathscr L}}
\newcommand{\sE}{{\mathscr E}}
\newcommand{\sF}{{\mathscr F}}
\newcommand{\sO}{{\mathscr O}}
\newcommand{\sV}{{\mathscr V}}
\newcommand{\sX}{{\mathscr X}}
\newcommand{\sY}{{\mathscr Y}}
\newcommand{\A}{{\mathbf A}}

\renewcommand{\C}{{\mathbf C}}

\newcommand{\F}{{\mathbf F}}
\renewcommand{\Im}{{\mathrm{Im}}}
\newcommand{\I}{{\mathrm I}}

\renewcommand{\P}{{\mathbf P}}
\newcommand{\Q}{{\mathbf Q}}

\newcommand{\Z}{{\mathbf Z}}
\newcommand{\Zl}{{\Z_{\ell}}}
\newcommand{\Ql}{{\Q_{\ell}}}
\newcommand{\Qp}{{\Q_{p}}}
\newcommand{\CH}{\mathrm{CH}}
\newcommand{\nr}{\mathrm{nr}}
\newcommand{\NS}{\mathrm{NS}}
\newcommand{\Pic}{\mathrm{Pic}}
\newcommand{\Br}{\mathrm{Br}}
\newcommand{\Spec}{\mathrm{Spec}}
\newcommand{\Div}{\mathrm{Div}}
\renewcommand{\phi}{\varphi}
\renewcommand{\emptyset}{\varnothing}
\newcommand{\red}{{\mathrm{red}}}
\newcommand{\Ker}{{\mathrm{Ker}}}
\newcommand{\Coker}{{\mathrm{Coker}}}

\newcommand{\otimeshat}{{\widehat{\otimes}}}
\newcommand{\mmu}{\boldsymbol{\mu}}
\newcommand{\Homrond}{\mathscr{H}\mkern-4muom}
\newcommand{\eq}[2]{\begin{equation}\label{#1}#2 \end{equation}}
\newcommand{\surj}{\twoheadrightarrow}
\newcommand{\inj}{\hookrightarrow}
\newcommand{\et}{\text{ét}}

\makeatletter\let\@wraptoccontribs\wraptoccontribs\makeatother

\date{June 22, 2013; revised on February~19, 2014}
\title[On the cycle class map for zero-cycles over local fields]{On the cycle class map for zero-cycles\\over local fields}

\author{H{\'e}l{\`e}ne Esnault}
\address{Freie Universit{\"a}t Berlin, Mathematik und Informatik,
Arnimallee~3, 14195 Berlin, Germany}
\email{esnault@math.fu-berlin.de}

\author{Olivier Wittenberg}
\address{D\'epartement de math\'ematiques et applications, \'Ecole normale sup\'erieure, 45~rue d'Ulm, 75230 Paris Cedex 05, France}
\email{wittenberg@dma.ens.fr}

\contrib[\kern3ex with an appendix \vbox to 1.05em{}by]{Spencer Bloch}
\address{5765 S. Blackstone Ave., Chicago, IL 60637, USA}
\email{spencer\_bloch@yahoo.com}

\subjclass[2010]{Primary 14C15; Secondary 14G20, 14J28, 14D06}

\thanks{The first author is supported by the the ERC Advanced Grant 226257, the Chaire d'Excellence 2011 of the Fondation Sciences Math\'ematiques de Paris and the Einstein Foundation.}

\begin{document}

\begin{abstract}
We study the Chow group of $0$\nobreakdash-cycles of smooth projective varieties over local and strictly local fields.
We prove in particular the injectivity of the cycle class map to integral $\ell$\nobreakdash-adic cohomology for a large class of
surfaces with positive geometric genus, over local fields of residue characteristic $\neq \ell$.
The same statement holds for semistable $K3$ surfaces defined over~$\C((t))$, but does not hold in general for surfaces over strictly local fields.
\end{abstract}

\maketitle

\section{Introduction}

Let~$X$ be a smooth projective variety over a field~$K$,
let $\CH_0(X)$ denote the Chow group of $0$\nobreakdash-cycles on~$X$ up to rational equivalence
and let $A_0(X) \subset \CH_0(X)$ be the subgroup of cycle classes of degree~$0$.

When~$K$ is algebraically closed, the group $A_0(X)$ is divisible and its structure as an abelian group is, conjecturally, rather well understood, thanks to
Roitman's theorem and to the Bloch--Beilinson--Murre conjectures.
A central tool for the study of~$A_0(X)$ over other types of fields is the cycle class map
\begin{align*}
\psi:\CH_0(X)/n\CH_0(X) \to H^{2d}_\et(X,\Z/n\Z(d))
\end{align*}
to \'etale cohomology,
where~$n$ denotes an integer invertible in~$K$ and $d=\dim(X)$.
The group $H^{2d}_\et(X,\Z/n\Z(d))$ is easier to understand, thanks to the Hochschild--Serre spectral sequence;
for instance, if~$K$ has cohomological dimension~$\leq 1$ and~$X$ is simply connected, then $H^{2d}_\et(X,\Z/n\Z(d))=\Z/n\Z$ and~$\psi$ may be interpreted as the degree map.

According to one of the main results of higher-dimensional unramified class field theory, due
to Kato and Saito~\cite{katosaito}, if~$K$ is a finite field, the group~$A_0(X)$
is finite and~$\psi$ is an isomorphism.
More recently, Saito and Sato~\cite{saitosato} have shown that
if~$K$ is
the quotient field of an excellent henselian discrete valuation ring with finite or
separably closed residue field,
the group~$A_0(X)$ is the direct sum of a finite group of order prime to~$p$
and a group divisible by all integers prime to~$p$
(see also \cite[Th\'eor\`eme~3.25]{ctbourbakisaitosato}).
In this case, however, the map~$\psi$ need not be either injective or surjective.
What Saito and Sato prove, instead, is the bijectivity of the analogous cycle class map for cycles of dimension~$1$ on regular models of~$X$ over
the ring of integers of~$K$ (see \cite[Theorem~1.16]{saitosato}).

Following a method initiated by Bloch~\cite{blochtorsion},
one may approach the torsion subgroup of~$A_0(X)$, as well as the kernel of~$\psi$,
with the help of algebraic K\nobreakdash-theory,
when~$X$ is a surface.
We refer to~\cite{ctcime} for a detailed account of this circle of ideas.
Strong results were obtained in this way for rational surfaces, and more generally for surfaces with geometric genus zero, over number fields, $p$\nobreakdash-adic fields, and
fields
of characteristic~$0$
and cohomological dimension~$1$ (see \cite{blochlectures}, \cite{blochrational}, \cite{ctsansequel}, \cite{ctk2}, \cite{ctraskindf}, \cite{saitocodim2}).
We note that over algebraically closed fields of characteristic~$0$, surfaces with geometric genus zero are those surfaces for which the Chow group~$A_0(X)$
should be representable, according to Bloch's conjecture (see~\cite[\textsection1]{blochlectures}).

The first theorem of this paper establishes
the injectivity of~$\psi$ for a large class of surfaces over local fields, when~$n$ is divisible enough and prime to the residue characteristic,
without any assumption on the geometric genus.  In~principle, this theorem should be applicable to all simply connected surfaces,
a generality in which the injectivity of~$\psi$ may not have been expected.
Before we state it, we set up some notation.
Let~$\sX$ be a regular proper flat scheme over an
an excellent henselian discrete valuation ring~$R$.
Let~$X=\sX \otimes_R K$ and $A=\sX \otimes_R k$ denote the generic fiber and the special fiber, respectively.
We assume the reduced special fiber $A_\red$ has simple normal crossings,
and write $\CH_0(X) \otimeshat \Zl = \varprojlim \CH_0(X) /\ell^n \CH_0(X)$.

\begin{thmintro}[Theorem~\ref{th:kfinite} and Remark~\ref{rk:leftkernel}]
\label{thm:A}
Assume the residue field~$k$ is finite and~$X$ is a surface whose Albanese variety has potentially good reduction.
If the irreducible components of~$A$ satisfy the Tate conjecture,
then for any~$\ell$ invertible in~$k$, the cycle class map
\begin{align}
\CH_0(X)\otimeshat \Zl \to H^4_\et(X,\Zl(2))
\end{align}
is injective.
Equivalently, the natural pairing $\CH_0(X) \times \Br(X) \to \Q/\Z$
is non-degenerate on the left modulo the maximal $\ell$\nobreakdash-divisible subgroup of $\CH_0(X)$.
\end{thmintro}

The assumption on the irreducible components of~$A$ holds
as soon as~$X$ has geometric genus zero, as well as in many examples
of nontrivial degenerations of surfaces with nonzero geometric genus (see~\textsection\ref{sec:surfacesoverpadic}).
Theorem~\ref{thm:A} is due to Saito~\cite{saitocodim2} when~$X$ is a surface with geometric genus zero over a $p$\nobreakdash-adic field.
An example of Parimala and Suresh~\cite{parimalasuresh} shows that the assumption on the Albanese variety cannot be removed.
Finally, we note that Theorem~\ref{thm:A} may be viewed as a higher-dimensional generalization
of Lichtenbaum--Tate duality for curves,
according to which
the natural pairing $\CH_0(X) \times \Br(X) \to \Q/\Z$
is non-degenerate if~$X$ is a smooth proper curve over a $p$\nobreakdash-adic field
(see~\cite{lichtenbaumduality}).

Our starting point for the proof of Theorem~\ref{thm:A} is the theorem of Saito and Sato alluded to above
about the cycle class map for $1$\nobreakdash-cycles on~$\sX$
\cite[Theorem~1.16]{saitosato}, which allows us to express the kernel of~$\psi$ purely in terms of the scheme~$A$ and of the cohomology of~$X$,
when~$k$ is either finite or separably closed
(Theorem~\ref{th:criterionsepclosed} and Theorem~\ref{th:criterionfinite}).
The dimension of~$X$ plays no role in this part of the argument;
an application to the study of $0$\nobreakdash-cycles on a cubic threefold over~$\Qp$ may be found in
Example~\ref{eq:cubicthreefold}.
Theorem~\ref{thm:A} is then obtained by analysing the various cohomology groups which appear in the resulting expression for~$\Ker(\psi)$.
More precisely,
in the situation of Theorem~\ref{thm:A}, we prove the stronger assertion that
the $1$\nobreakdash-dimensional cycle class map
$\psi_{1,A}:\CH_1(A) \otimeshat \Zl \to H^4_A(\sX,\Zl(2))$
to integral $\ell$\nobreakdash-adic \'etale homology is surjective.
This provides a geometric explanation for the assumption
that the Albanese variety of~$X$ have potentially good reduction,
a condition which first appeared in~\cite{saitocodim2}
and which turns out to be essential for the surjectivity of~$\psi_{1,A}$ to hold
(see~Lemma~\ref{lem:albgoodred} and~\textsection\ref{subsec:aremarkpsi1A}).

When the residue field~$k$ is separably closed instead of finite,
the arguments used in the proof of Theorem~\ref{thm:A} fail in several places.  They still lead to the following statement, which
may also be deduced from results of Colliot-Th\'el\`ene and Raskind~\cite{ctraskind} (see~\textsection\ref{sec:surfacesoverC((t))} for comments on this point).

\begin{thmintro}[Theorem~\ref{th:H2alg-ksepclosed} and Remark~\ref{rk:unramifiedh3}]
\label{thm:B}
Assume~$k$ is separably closed and~$K$ has characteristic~$0$.  If~$X$ is a surface with geometric genus zero,
then for any~$\ell$ invertible in~$k$, the cycle class map
\begin{align}
\label{eq:cycleclassmap}
\CH_0(X)\otimeshat \Zl \to H^4_\et(X,\Zl(2))
\end{align}
is injective.  If in addition~$X$ is simply connected, then $A_0(X)$ is divisible by~$\ell$
and the unramified cohomology group $H^3_\nr(X,\Ql/\Zl(2))$ vanishes.
\end{thmintro}

This leaves open the question of the injectivity of the cycle class map~\eqref{eq:cycleclassmap} when~$k$ is separably closed
and~$X$ is a surface with positive geometric genus
over~$K$.  In this situation, the $1$\nobreakdash-dimensional cycle class map~$\psi_{1,A}$ is far from being surjective.
Building on the work of Kulikov, Persson, Pinkham
\cite{kulikov} \cite{perssonpinkham},
 and of Miranda and Morrison
\cite{mirandamorrison}, we nevertheless give a positive answer for semistable~$K3$ surfaces over $\C((t))$.

\begin{thmintro}[Theorem~\ref{th:semiK3}]
\label{thm:C}
Let~$X$ be a~$K3$ surface over $\C((t))$.
If~$X$ has semistable reduction, the group $A_0(X)$ is divisible.
\end{thmintro}

The proof of Theorem~\ref{thm:C} hinges on the precise knowledge of the combinatorial structure of a degeneration of~$X$.
It would go through over the maximal unramified extension of a $p$\nobreakdash-adic field, as far as prime-to\nobreakdash-$p$ divisibility is concerned,
if similar knowledge were available.  This is in marked contrast with the situation over $p$\nobreakdash-adic fields,
where such knowledge is not necessary for the proof of Theorem~\ref{thm:A}.

In the final section of this paper, with the help of Ogg--Shafarevich theory and of a construction due to Persson,
we show that the hope for a statement analogous to Theorem~\ref{thm:A} over the quotient field of a strictly henselian excellent discrete valuation ring
is in fact too optimistic, even over the maximal unramified extension of a $p$\nobreakdash-adic field.

\begin{thmintro}[Theorem~\ref{th:counterexample}]
\label{thm:D}
There exists a simply connected smooth projective surface~$X$ over $\C((t))$ such that $A_0(X)/2A_0(X)=\Z/2\Z$.
For infinitely many prime numbers~$p$, there exists a
simply connected smooth projective surface~$X$ over the maximal unramified extension of a $p$\nobreakdash-adic field
such that $A_0(X)/2A_0(X)=\Z/2\Z$.
\end{thmintro}

By a theorem of Kato and Saito~\cite[\textsection10]{katosaito},
the group~$A_0(X)$
is finite
for any smooth projective variety~$X$ over a finite field
and it vanishes if~$X$ is simply connected.
Theorem~\ref{thm:D} answers in the negative
the question
whether the same result might also
hold over the quasi-finite field~$\C((t))$,
for the quotient of~$A_0(X)$ by its maximal divisible subgroup.

In higher dimension, the statement of Theorem~\ref{thm:A} is known to fail for rationally connected threefolds over local fields (see~\cite[Proposition~6.2]{parimalasuresh};
this example is even unirational over~$\Q_3$, by \cite[Corollary~1.8]{kollarlocal}).
It would be interesting to find an example of a rationally connected threefold~$X$ over~$\C((t))$ such that $A_0(X)\neq 0$.
Some constraints on the possible degenerations of such an~$X$ may be gathered from the proof of Theorem~\ref{thm:B} given below.

\bigskip
\emph{Acknowledgements.}
We are grateful to Spencer Bloch for his interest and for
contributing to our article in the form of an appendix, in which he simplifies the proof of~\cite[Theorem~1.16]{saitosato}.
We~thank Pierre Berthelot, Jean-Louis Colliot-Th\'el\`ene,
Bruno Kahn, Shuji Saito, and Takeshi Saito for discussions on this or related topics.
In addition we~thank the referees for their work and their helpful comments.

\bigskip
\emph{Notation.}
If~$M$ is an abelian group and~$\ell$ is a prime number, we write
$M\otimeshat\Zl$ for the $\Zl$\nobreakdash-module $\varprojlim M/\ell^nM$.  We say that~$M$ is \emph{divisible by~$\ell$} if $M/\ell M=0$,
and that it is \emph{divisible} if it is divisible by all prime numbers.
Unless otherwise specified, all cohomology groups are Galois or \'etale cohomology groups.
Cohomology with coefficients in~$\Zl$ stands for the inverse limit of the corresponding cohomology groups with coefficients in~$\Z/\ell^n\Z$,
and will only be considered in situations in which the latter groups are all finite.
If~$V$ is a proper variety over a field~$k$, we denote by $\NS(V)$ the quotient of $\Pic(V)$ by algebraic equivalence,
and we denote by~$A_0(V)$ the group of $0$\nobreakdash-cycles of degree~$0$ on~$V$ up to rational equivalence, \emph{i.e.}, the kernel
of the degree map $\CH_0(V)\to \Z$.  We let $\bar V=V \otimes_k \bar k$, where~$\bar k$ denotes a separable closure of~$k$.
Finally, we say that a variety~$V$ over a field~$k$ has \emph{simple normal crossings} if it is reduced
and if for any finite collection of pairwise distinct irreducible components $V_1,\dots,V_n$ of~$V$,
the scheme-theoretic intersection $\bigcap_{1\leq i\leq n} V_i$ is smooth over~$k$ and has codimension $n-1$ in~$V$ at every point.

\section{A criterion for the injectivity of the \texorpdfstring{$\ell$}{ℓ}-adic cycle class map}
\label{sec:criterion}

Let~$\sX$ denote an irreducible regular projective flat scheme over an excellent henselian discrete valuation ring~$R$, with special fiber~$A$ and generic fiber~$X$.
Let~$\ell$ be a prime number invertible in~$R$. The goal of this section is to establish
a criterion, in terms of~$A$ and of the cohomology of~$X$, for the $\ell$\nobreakdash-divisibility, up to rational equivalence,
of homologically trivial zero-cycles on~$X$, when~$k$ is either finite or separably closed.
This criterion is derived from the main theorem of Saito and Sato~\cite{saitosato} (see also Appendix~\ref{appendix}).

We denote by~$K$ the quotient field of~$R$.
Let $d=\dim(X)$ and let $(A_i)_{i \in I}$ denote the family of irreducible components of~$A$.

\begin{thm}
\label{th:criterionsepclosed}
Assume~$k$ is separably closed.
The kernel of the cycle class map
$$\CH_0(X) \otimeshat \Zl \longrightarrow H^{2d}(X,\Zl(d))$$
is a finite group, canonically Pontrjagin dual to the homology group of the complex
\begin{align}
\label{eq:complexsepclosed}
\xymatrix{
H^1(X, \Ql/\Zl(1)) \ar[r] & H^2_A(\sX,\Ql/\Zl(1)) \ar[r] & \displaystyle\bigoplus_{i \in I} \frac{H^2(A_i,\Ql/\Zl(1))}{\CH_1(A_i)^\perp}\rlap{\text{.}}
}
\end{align}
\end{thm}

\begin{thm}
\label{th:criterionfinite}
Assume~$k$ is finite.
The kernel of the cycle class map $$\CH_0(X) \otimeshat \Zl \longrightarrow H^{2d}(X,\Zl(d))$$ is a finite group, canonically Pontrjagin dual to the homology group of the complex
\begin{align}
\label{eq:complexfinite}
\xymatrix{
H^2(X, \Ql/\Zl(1)) \ar[r] & H^3_A(\sX,\Ql/\Zl(1)) \ar[r] & \displaystyle\bigoplus_{i \in I} \frac{H^3(A_i,\Ql/\Zl(1))}{\CH_1(A_i)^\perp}\rlap{\text{.}}
}
\end{align}
\end{thm}

Let us explain how the complexes~\eqref{eq:complexsepclosed} and~\eqref{eq:complexfinite} are defined.

For any $i \in I$,
the cycle class map $\CH_1(A_i)\to H^{2d}_{A_i}(\sX,\Zl(d))$
and cup-product
give rise together to
a natural pairing
$H^m(A_i,\Ql/\Zl(1)) \times \CH_1(A_i) \to \Ql/\Zl$,
where $m=2$ if~$k$ is separably closed and $m=3$ if~$k$ is finite
(see Proposition~\ref{prop:dualityAsX} below,
and see~\cite[\textsection1]{saitosato} for the definition of the cycle class map).
The notation $\CH_1(A_i)^\perp$ refers to the left kernel of this pairing.  Concretely
$\CH_1(A_i)^\perp$ is the group of those classes in $H^m(A_i,\Ql/\Zl(1))$ whose restriction to $H^m(C,\Ql/\Zl(1))$ vanishes for any irreducible curve~$C$ lying on~$A_i$.

In both~\eqref{eq:complexsepclosed} and~\eqref{eq:complexfinite}, the first map is the natural map, while the second map is the composition
of the natural map $H^m_A(\sX,\Ql/\Zl(1)) \to H^m(\sX,\Ql/\Zl(1))$ with the sum of the restriction maps $H^m(\sX,\Ql/\Zl(1)) \to H^m(A_i,\Ql/\Zl(1))$
and the quotient maps $H^m(A_i,\Ql/\Zl(1)) \to H^m(A_i,\Ql/\Zl(1))/\CH_1(A_i)^\perp$.

\begin{rmks}
\label{rk:complexsepclosed}
(i)
It follows from purity in \'etale cohomology
(see~\cite{fujiwara}) that
there is a canonical isomorphism $H^2_A(\sX,\Lambda(1)) = \Lambda^I$
for $\Lambda=\Z/n\Z$ with~$n$ invertible in~$R$.
For the reader's convenience and for lack of an adequate reference, we provide an elementary proof.
Let $\iota:A \hookrightarrow \sX$ and $j:X\hookrightarrow \sX$
denote the canonical immersions.
It~is a general fact that $\iota_*R^2\iota^!\Lambda=R^1j_*\Lambda$ (see \cite[p.~107]{freitagkiehl}).
On the other hand, the Kummer exact sequence and the vanishing of~$R^1j_*\Gm$ (see \cite[Chapter~III, Proposition~4.9]{milne})
imply together that $R^1 j_* \Lambda(1) = (j_*\Gm)\otimes_\Z \Lambda$.
Finally, we have $(j_*\Gm)\otimes_\Z \Lambda=D \otimes_\Z \Lambda$,
where~$D$ denotes the sheaf of Cartier divisors on~$\sX$ supported on~$A$ (see \cite[Exp.~IX, \textsection3.3.1]{sga43}).
Thus $H^2_A(\sX,\Lambda(1))=H^0(\sX,\iota_*R^2\iota^!\Lambda(1))=H^0(\sX,D \otimes_\Z \Lambda)=\Lambda^I$.

(ii)
Assume~$k$ is separably closed.
In this case~$K$ has cohomological dimension~$1$
and the Hochschild--Serre spectral sequence presents the group
$H^1(X,\Ql/\Zl(1))$ as an extension of $H^0(K,H^1(\bar X,\Ql/\Zl(1)))$
by $H^1(K,H^0(\bar X,\Ql/\Zl(1)))=\Ql/\Zl$,
where~$\bar X = X \otimes_K \bar K$ and where~$\bar K$ is a separable closure of~$K$.
On the other hand, 
by the localization exact sequence,
the first map of~\eqref{eq:complexsepclosed} vanishes on the image of $H^1(\sX,\Ql/\Zl(1))$
in $H^1(X,\Ql/\Zl(1))$.
Lastly, we have $H^1(\sX,\Ql/\Zl(1))=H^1(A,\Ql/\Zl(1))$ according to the proper base change theorem.
These three facts, combined with Remark~\ref{rk:complexsepclosed}~(i),
show that
in the statement of Theorem~\ref{th:criterionsepclosed},
the complex~\eqref{eq:complexsepclosed} may be replaced
with the more explicit complex
\begin{align}
\label{eq:complexsepclosedbis}
\xymatrix{
\Phi \ar[r] & \displaystyle \Coker\Big(\Ql/\Zl \xrightarrow{\;\alpha\;} (\Ql/\Zl)^I\Big) \ar[r]^(.56)\delta & \displaystyle\bigoplus_{i\in I} \frac{\NS(A_i) \otimes_\Z \Ql/\Zl}{\CH_1(A_i)^\perp}\rlap{\text{,}}
}
\end{align}
where $\Phi=\Coker\big(H^1(A,\Ql/\Zl(1))\to H^0(K,H^1(\bar X,\Ql/\Zl(1)))\big)$ is a finite group which measures the ``defect of the local invariant cycle theorem with
torsion coefficients'',
where
for any $i\in I$,
the map~$\alpha$,
on the $i$th component, is multiplication by the multiplicity of~$A_i$ in~$A$,
and where for any $i \in I$
and any $\lambda=(\lambda_j)_{j \in I} \in (\Ql/\Zl)^I$,
the $i$th component of $\delta(\lambda)$
is the class of~$\sO_{\sX}(A_j)|_{A_i} \otimes \lambda_j$.
We note that~$\Phi$ vanishes as soon as~$H^1(\bar X,\Z/\ell\Z)=0$.

(iii)
When~$k$ is separably closed and $H^1(\bar X,\Z/\ell\Z)=0$, the Hochschild--Serre spectral sequence implies that
the kernel of the cycle class map $\CH_0(X)\otimeshat \Zl \to H^{2d}(X,\Zl(d))$ coincides with $A_0(X) \otimeshat\Zl$,
where $A_0(X)=\Ker(\deg:\CH_0(X)\to \Z)$.
Hence, in this case, Theorem~\ref{th:criterionsepclosed} states that the finite group $A_0(X) \otimeshat \Zl$ is canonically Pontrjagin
dual to the kernel of~$\delta$.

(iv)
Assume~$k$ is separably closed.
Since $\sO_{\sX}(A_i)|_{A_j}=\sO_{A_j}(A_i \cap A_j)$ whenever $i,j\in I$ are distinct
and since, for any~$i$, any element of $\Coker(\alpha)$ may be represented by a family $(\lambda_j)_{j \in I} \in (\Ql/\Zl)^I$ with $\lambda_i=0$,
the explicit description of the map~$\delta$ given above makes it clear that the group $\Ker(\delta)$ depends only on the scheme~$A$.
Thus, as a consequence of Theorem~\ref{th:criterionsepclosed}, if~$\sX'$ is another proper regular flat $R$\nobreakdash-scheme whose special fiber is isomorphic to~$A$,
and if we assume that $H^1(\bar X,\Z/\ell\Z)=H^1(\bar X',\Z/\ell\Z)=0$,
where $X'=\sX'\otimes_R K$,
then
the groups $A_0(X) \otimeshat \Zl$ and $A_0(X') \otimeshat \Zl$
are isomorphic.
\end{rmks}

\begin{examples}
\label{ex:firstex}
(i) If $\sX$ is smooth over~$R$, 
the natural map
\begin{align*}
H^m_A(\sX,\Ql/\Zl(1)) \to H^m(\sX,\Ql/\Zl(1))
\end{align*}
identically
vanishes for any $m \geq 0$, since
its composition with the
purity isomorphism
$H^{m-2}(A,\Ql/\Zl) \isoto H^m_A(\sX,\Ql/\Zl(1))$
(see \cite[Theorem~2.1.1]{fujiwara})
may be interpreted as cup-product with the cycle class of the divisor~$A$ on~$\sX$.
It follows that the leftmost maps of~\eqref{eq:complexsepclosed} and~\eqref{eq:complexfinite} are surjective in this case.
Thus, when~$X$ has good reduction and the prime number~$\ell$ is invertible in~$R$,
the group $A_0(X)$ is divisible by~$\ell$
if~$k$ is separably closed
and
the cycle class map $\CH_0(X) \otimeshat \Zl \to H^{2d}(X,\Zl(d))$ is injective
if~$k$ is finite.
For separably closed~$k$ this was noted already in \cite[Corollary~0.10]{saitosato}.

(ii) If~$\sX$ is smooth over~$R$ and~$k$ is algebraically closed of characteristic $p>0$,
the group~$A_0(X)$ need not, however, be divisible by~$p$.  An example of this phenomenon is given by any elliptic curve with
good ordinary reduction.
Indeed, let~$E$ be such an elliptic curve, let~$\sE$ denote its smooth proper model over~$R$,
let $\pi:\sE \to \sX$ denote the quotient by the unique subgroup scheme isomorphic to~$\mmu_p$ over~$R$,
and let $X=\sX \otimes_R K$.
The group $A_0(X)/pA_0(X)$ is then infinite, as it surjects onto $H^1_{\mathrm{fppf}}(R,\mmu_p)=R^*/R^{*p}$ via the boundary map associated to~$\pi$,
according to \cite[Th\'eor\`eme~11.7]{grothendieckbr3}.

(iii) Assume that~$k$ is separably closed and that~$A$ is reduced and has two irreducible components~$A_1$ and~$A_2$, which are smooth and meet transversally along a smooth irreducible
subvariety~$D$ of dimension~$d-1$.
Assume moreover that $H^1(\bar X,\Z/\ell\Z)=0$ (\emph{e.g.}, that~$X$ is simply connected).
Then, according to Theorem~\ref{th:criterionsepclosed} and to Remarks~\ref{rk:complexsepclosed}~(ii) and~(iii), the dimension of $A_0(X)/\ell A_0(X)$
as a vector space over~$\Z/\ell\Z$ is~$\leq 1$, and it is~$1$ if and only if
for any $i \in \{1,2\}$ and any irreducible curve~$C$ lying on~$A_i$, the intersection number $(C\cdot D)$, computed on~$A_i$, is divisible by~$\ell$.
\end{examples}

\subsection{Reminders on duality}
\label{subsec:duality}

For lack of an adequate reference, we gather in this paragraph various duality results that will be used in the proof
of Theorems~\ref{th:criterionsepclosed} and~\ref{th:criterionfinite}.

\begin{prop}
\label{prop:dualityoverfield}
Let~$V$ be an irreducible, proper and smooth variety over a field~$F$, and let $d=\dim(V)$.
Let~$\sF$ be a constructible \'etale sheaf of $\Z/n\Z$\nobreakdash-modules on~$V$,
where~$n$ is an integer invertible in~$F$, and let $\sF^\vee(j)=\Homrond(\sF,\Z/n\Z(j))$ for $j\in \Z$.
Let $i \in \Z$.
\begin{enumerate}
\item If~$F$ is separably closed, then $H^i(V,\sF)$ is canonically dual to
$H^{2d-i}(V,\sF^\vee(d))$.
\smallskip
\item If~$F$ is finite, then
$H^i(V,\sF)$
is canonically dual to
$H^{2d+1-i}(V,\sF^\vee(d))$.
\smallskip
\item If~$F$ is the quotient field of a strictly henselian discrete valuation ring in which~$n$ is invertible,
then $H^i(V,\sF)$ is canonically dual to $H^{2d+1-i}(V,\sF^\vee(d+1))$.
\smallskip
\item If~$F$ is the quotient field of a henselian discrete valuation ring~$R$ in which~$n$ is invertible
and if the residue field of~$R$ is finite,
then
$H^i(V,\sF)$ is canonically dual to $H^{2d+2-i}(V,\sF^\vee(d+1))$.
\end{enumerate}
\end{prop}

The first statement of Proposition~\ref{prop:dualityoverfield} is Poincar\'e duality.  The other three statements are obtained by combining Poincar\'e duality with
well-known duality theorems for the Galois cohomology of the fields under consideration.   A proof of~(4) may be found
in \cite[Lemma~2.9]{saitolakelouise}.  Over finite or strictly local fields, the same arguments lead to~(2) and~(3).

\begin{prop}
\label{prop:dualityAsX}
Let~$\sX$, $R$, $k$, $A$, $X$ be as at the beginning of~\textsection\ref{sec:criterion}.
Let $\Lambda=\Z/n\Z$ where~$n$ is invertible in~$R$.  Assume~$k$ is either separably closed or finite.
Let $m=2$ if~$k$ is separably closed and $m=3$ if~$k$ is finite.
There is a canonical isomorphism $H^{2d+m}_A(\sX,\Lambda(d+1))=\Lambda$. Together with cup-product it induces a perfect pairing
of finite abelian groups
\begin{align}
H^i(A,\Lambda(j)) \times H^{2d+m-i}_A(\sX,\Lambda(d+1-j)) \to \Lambda
\end{align}
for any $i,j\in \Z$.
\end{prop}

\begin{proof}
Let us fix an embedding $e:\sX \hookrightarrow \P^N_R$ and consider the commutative diagram
\begin{align}
\begin{aligned}
\xymatrix{
A \ar[r]^\iota \ar[d]^(.45)g & \sX \ar[d]^(.45)f \ar[r]^e & \P^N_R \ar[dl]^(.4){\mkern-3muq} \\
\Spec(k) \ar[r]^h & \Spec(R)\rlap{\text{.}}
}
\end{aligned}
\end{align}
Let $c=N-d$.
In the derived categories of \'etale sheaves of $\Lambda$\nobreakdash-modules
on~$\Spec(k)$, on~$\sX$ and on~$\P^N_R$,
there are canonical isomorphisms $Rh^!\Lambda=\Lambda(-1)[-2]$,
$Re^!\Lambda=\Lambda(-c)[-2c]$
and $Rq^!\Lambda=\Lambda(N)[2N]$,
by purity and Poincar\'e duality
(see \cite[Theorem~2.1.1]{fujiwara} and \cite[Exp.~XVIII, Th\'eor\`eme~3.2.5]{sga43} respectively).
As $q\circ e\circ \iota=h\circ g$, we deduce that $Rg^!\Lambda=R\iota^!\Lambda(d+1)[2d+2]$,
and hence that
\begin{align}
\label{eq:Rgdual}
R\Homrond(Rg_*\Lambda,\Lambda)=Rg_*R\iota^!\Lambda(d+1)[2d+2]\rlap{\text{.}}
\end{align}
Proposition~\ref{prop:dualityAsX} follows from this remark and from the fact that $H^i(k,\sF)$ is canonically dual
to $H^{m-2-i}(k,R\Homrond(\sF,\Lambda))$ for any bounded complex~$\sF$ of \'etale sheaves of $\Lambda$\nobreakdash-modules on~$\Spec(k)$
(see \cite[Chapter~I, Example~1.10]{milneadt}).
\end{proof}

\begin{rmk}
\label{rk:pairingscompatible}
The pairings
which appear in the statements of
Proposition~\ref{prop:dualityoverfield}
and Proposition~\ref{prop:dualityAsX}
are compatible with the maps stemming from the localization exact sequence.
In particular, in the situation of Proposition~\ref{prop:dualityAsX},
the natural maps
\begin{align*}
H^i(A,\Lambda(j))=H^i(\sX,\Lambda(j)) \to H^i(X,\Lambda(j))
\end{align*}
(where the first equality comes from the proper base change theorem)
and
\begin{align*}
H^{2d+m-i-1}(X,\Lambda(d+1-j))\to
H^{2d+m-i}_A(\sX,\Lambda(d+1-j))
\end{align*}
form a pair of adjoint maps.
So do the natural maps
$$H^i(X,\Lambda(j)) \to H^0(K,H^i(\bar X,\Lambda(j)))$$
and
$$H^1(K,H^{2d-i}(\bar X,\Lambda(d+1-j))) \to H^{2d+1-i}(X,\Lambda(d+1-j))$$
coming from the Hochschild--Serre spectral sequence
when~$k$ is separably closed.
These two compatibilities follow immediately from the definitions of the pairings involved.
\end{rmk}

\subsection{Proof of Theorems~\ref{th:criterionsepclosed} and~\ref{th:criterionfinite}}
\label{subsec:proofcriterion}

We assume~$k$ is either finite or separably closed.
The localization exact sequences for Chow groups and for \'etale cohomology, the cycle class maps for $0$\nobreakdash-cycles on~$X$, for $1$\nobreakdash-cycles
on~$\sX$, for $1$\nobreakdash-cycles on~$A$ and for $1$\nobreakdash-cycles on the $A_i$'s fit together in a commutative diagram
\begin{align}
\begin{aligned}
\label{cd:largediagram}
\xymatrix@C=-0.5em{
\smash[b]{\displaystyle\bigoplus_{i \in I} \CH_1(A_i)}/\ell^n \ar@{->>}[dr] \ar[rr]^(.45){\bigoplus \psi_{1,A_i}} && \smash[b]{\displaystyle\bigoplus_{i\in I} H^{2d}_{A_i}}(\sX,\Z/\ell^n(d)) \ar[dr] \\
&\CH_1(A)/\ell^n \ar[rr]^(.45){\psi_{1,A}} \ar[d] && H^{2d}_A(\sX,\Z/\ell^n(d)) \ar[d] \\
&\CH_1(\sX)/\ell^n \ar[rr]_(.45)\sim^(.45){\psi_{1,\sX}} \ar@{->>}[d] && H^{2d}(\sX,\Z/\ell^n(d)) \ar[d] \\
&\CH_0(X)/\ell^n \ar[rr]^(.45){\psi_{0,X}} && H^{2d}(X,\Z/\ell^n(d))
}
\end{aligned}
\end{align}
for each $n \geq 1$
(see
\cite[\textsection20.3]{fulton},
\cite[\textsection1]{saitosato}).
According to the main theorem of Saito and Sato~\cite{saitosato}, the map $\psi_{1,\sX}$ is an isomorphism
(see Theorem~\ref{app:thm1} and Remark~\ref{app:rmk}).
In addition, the bottom left vertical map of~(\ref{cd:largediagram}) is onto.
Thus, the kernel of $\psi_{0,X}$ identifies with the homology group of the complex
\begin{align}
\label{eq:complexextracted}
\xymatrix{
\displaystyle\bigoplus_{i \in I} \CH_1(A_i)/\ell^n \ar[r] & H^{2d}(\sX,\Z/\ell^n(d)) \ar[r] & H^{2d}(X,\Z/\ell^n(d))
}
\end{align}
extracted from this diagram.
Let $m=2$ if~$k$ is separably closed and $m=3$ if~$k$ is finite.
As the first map of~\eqref{eq:complexextracted} factors through $\bigoplus H^{2d}_{A_i}(\sX,\Z/\ell^n(d))$ via the cycle class map $\bigoplus \psi_{1,A_i}$,
applying Pontrjagin duality to~\eqref{eq:complexextracted} yields the complex
\begin{align}
\label{eq:complexextractdual}
\xymatrix{
H^{m-1}(X,\Z/\ell^n(1)) \ar[r] & H^m_A(\sX,\Z/\ell^n(1)) \ar[r] & \displaystyle\bigoplus_{i \in I} \frac{H^m(A_i,\Z/\ell^n(1))}{\CH_1(A_i)^\perp}
}
\end{align}
(see Propositions~\ref{prop:dualityoverfield} and~\ref{prop:dualityAsX}).
The complexes~\eqref{eq:complexsepclosed} and~\eqref{eq:complexfinite} are obtained from~\eqref{eq:complexextractdual} by passing to the direct limit over~$n$.
On the other hand, the kernel of the cycle class map $\CH_0(X) \otimeshat \Zl \to H^{2d}(X,\Zl(d))$ is the inverse limit, over~$n$, of $\Ker(\psi_{0,X})$.
This concludes the proof of Theorems~\ref{th:criterionsepclosed} and~\ref{th:criterionfinite}.

\subsection{Further examples}

In \textsection\textsection\ref{sec:surfacesoverpadic}--\ref{sec:counterexample},
we shall apply Theorems~\ref{th:criterionsepclosed} and~\ref{th:criterionfinite} to establish various results on the injectivity of the cycle class map
for zero-cycles on varieties defined over local or strictly local fields.  The next three examples do not fall into the scope of these results.
We include them to illustrate the applicability of the above criteria.

\begin{example}
\label{ex:isotrivialK3}
Let $K=\C((t))$ and let $X\subset \P^3_K$ denote the~$K3$ surface defined by
\begin{align}
x_0^4+x_1^4+t^2(x_2^4+x_3^4)=0\rlap{\text{.}}
\end{align}
Let $\sX_0 \subset \P^3_R$ denote the projective scheme over $R=\C[[t]]$ defined by the same
equation, and let $\sX$ be the blow-up of~$\sX_0$ along the closed subscheme defined by the homogeneous ideal
$(x_0^3, x_1^3, x_0^2x_1, x_0x_1^2, x_0t, x_1t, t^2)$.  One checks that~$\sX$ is regular, and that
its special fiber~$A$ may be written,
as a divisor on~$\sX$, as $A=2S+P_1+\dots+P_4+Q_1+\dots+Q_4$, where the~$P_i$'s are pairwise disjoint
projective planes, the $Q_i$'s are pairwise disjoint smooth projective rational surfaces, and~$S$ is
a smooth~$K3$ surface endowed with an elliptic fibration $\pi:S \to \P^1_\C$ with four singular fibers,
each of type~$\I_0^*$ in Kodaira's notation~\cite{kodaira}.  Moreover,
the schemes $P_i \cap S$ (resp.~$Q_i \cap S$) are smooth irreducible rational curves
with self-intersection~$1$ on~$P_i$ (resp.~on~$Q_i$),
and we have $P_i \cap Q_j=\emptyset$ for all $i,j$.
The curves $P_i \cap S$ are the four $2$\nobreakdash-torsion sections of~$\pi$, and the curves~$Q_i \cap S$ are the irreducible
components of multiplicity~$2$ of the singular fibers of~$\pi$.
With this geometric description of the special fiber of a regular model of~$X$ in hand, it is now an exercise to check the exactness of the complex~\eqref{eq:complexsepclosedbis} for any prime number~$\ell$.
Note that $\Phi=0$ since~$X$ is simply connected.
By Theorem~\ref{th:criterionsepclosed} and Remarks~\ref{rk:complexsepclosed}~(ii) and~(iii), it follows that the group $A_0(X)$ is divisible.
\end{example}

\begin{example}
\label{ex:isotrivialdegn}
Let $K=\C((t))$
and let $X\subset \P^3_K$ denote the simply connected isotrivial surface defined by
\begin{align}
x_0^n+tx_1^n+t^2x_2^n+t^3x_3^n=0\rlap{\text{.}}
\end{align}

Let $K_e=\C((t^{1/e}))$.
For $n \leq 3$, the surface~$X$ is (geometrically) rational, which implies that $A_0(X)=0$ according to \cite[Theorem~A~(iv)]{ctk2}
(see also Example~\ref{ex:n=3} below).
For $n=4$, the index of~$X$ over~$k$,
\emph{i.e.}, the greatest common divisor of the degrees of the closed points of~$X$, is equal to~$2$ (see \cite[Proposition~4.4]{ELW}).
As every finite extension of~$K$ is cyclic, it follows that~$K_2$ embeds into the residue field of any closed point of~$X$. Therefore
the norm map
$A_0(X \otimes_K K_2) \to A_0(X)$ is surjective.   On the other hand,
the surface $X \otimes_K K_2$ over~$K_2$ is, after an obvious change of notation, the one considered in Example~\ref{ex:isotrivialK3};
thus $A_0(X \otimes_K K_2)$ is divisible. We conclude that for $n=4$, the group $A_0(X)$ itself is divisible.
Finally, let us note that a similar (though simpler) argument shows that
for $n=5$, or more generally for any~$n$ prime to~$6$, the group $A_0(X)$ is again divisible.
Indeed, on the one hand, if~$n$ is prime to~$6$, then,
by \cite[Proposition~4.4]{ELW}, the index of~$X$ over~$K$ is equal to~$n$;
and on the other hand, the group $A_0(X \otimes_K K_n)$ is divisible,
as the surface $X  \otimes_K K_n$ has good reduction over~$K_n$.

All these observations support the guess that the group $A_0(X)$ may in fact be divisible for any value of~$n$.
\end{example}

\begin{example}
\label{ex:cubicthreefold}
Let $K=\Qp$ for some prime number $p \neq 3$ and let $X \subset \P^4_K$ denote the cubic threefold defined by
\begin{align}
\label{eq:cubicthreefold}
x_0^3 + x_1^3 + x_2^3 + px_3^3 + p^2x_4^3 = 0\rlap{\text{.}}
\end{align}
Theorem~\ref{th:criterionfinite} may be used to show that $A_0(X)=0$. This answers
a question of Colliot-Th\'el\`ene \cite[\textsection12.6.2]{ctdegenerescences}.

We briefly sketch the structure of the argument.
The smooth cubic hypersurface~$X$ acquires good reduction over $K'=\Qp(\sqrt[3]p)$.
By a theorem of Koll\'ar and Szab\'o, it follows that $A_0(X \otimes_KK')=0$
(see \cite[Theorem~5]{kollarszabo}, \cite[Example~4.5]{derenthalkollar}).
In particular $A_0(X)$ is a $3$\nobreakdash-torsion group.
On the other hand, the group $H^6(X,\Z_3(3))$ is torsion-free, by the weak Lefschetz theorem.
Thus the vanishing of $A_0(X)$ is equivalent to the injectivity of the cycle class map
$\CH_0(X) \otimeshat \Z_3 \to H^6(X,\Z_3(3))$, which in turn is equivalent to the exactness
of the complex~\eqref{eq:complexfinite}.
Starting from the hypersurface defined by~\eqref{eq:cubicthreefold} in~$\P^4_{\Z_p}$,
one obtains a proper regular model~$\sX$ of~$X$ by three successive blow-ups of the singular locus.
The exactness of~\eqref{eq:complexfinite} may then be proved by a careful analysis of the geometry and the cohomology of the irreducible components of the special fiber
of~$\sX$.

It remains an open question whether the Chow group of zero-cycles of degree~$0$ vanishes for any smooth cubic hypersurface of dimension~$\geq 3$ over a $p$\nobreakdash-adic field.
We recall that for a cubic surface over a local field,
the Brauer group may be responsible for an obstruction
to the vanishing of~$A_0(X)$ (see, \emph{e.g.}, \cite[Exemple~2.8]{ctsaito}),
whereas
cubic hypersurfaces of dimension~$\geq 3$, on the other hand, satisfy $\Br(X)=\Br(K)$ (see~\cite[Appendix~A]{poonenvoloch}).
\end{example}

\section{Surfaces over \texorpdfstring{$p$-adic}{𝑝-adic} fields}
\label{sec:surfacesoverpadic}

A theorem of Saito~\cite[Theorem~A]{saitocodim2} asserts that if~$X$ is a smooth projective surface over a $p$\nobreakdash-adic field,
if $H^2(X,\sO_X)=0$,
and
if the Albanese variety of~$X$ has potentially good reduction, then for any prime~$\ell$,
the cycle class map $\CH_0(X) \otimeshat \Zl \to H^4(X, \Zl(2))$ is injective on the torsion of $\CH_0(X) \otimeshat \Zl$.
In this section, we prove that the assumption that $H^2(X,\sO_X)=0$ may be dropped from Saito's theorem,
when $\ell \neq p$
and the irreducible components of the special fiber of a regular model of~$X$
satisfy the Tate conjecture.  It~should be noted that
the assumption on the Albanese variety, on the other hand, cannot be dispensed with:
Parimala and Suresh
have given an example of a smooth proper surface~$X$ over~$\Q_3$, with $H^2(X,\sO_X)=0$,
such that the cycle class map $\CH_0(X) \otimeshat \Z_2 \to H^4(X,\Z_2(2))$ is not injective (see \cite[Example~8.2 and Proposition~7.5]{parimalasuresh}).

\begin{thm}
\label{th:kfinite}
With the notation of \textsection\ref{sec:criterion}, assume that $d=2$,
that~$k$ is finite,
that the reduced special fiber $A_\red$ of~$\sX$ has simple normal crossings, and that the surfaces~$A_i$
satisfy the Tate conjecture.
Let~$\ell$ be a prime number invertible in~$k$.
If the Albanese variety of~$X$ has potentially good reduction,
then the cycle class map $\CH_0(X) \otimeshat \Zl \to H^{2d}(X, \Zl(d))$ is injective.
\end{thm}

The hypothesis that the surfaces~$A_i$ satisfy the Tate conjecture automatically holds if~$K$ has characteristic~$0$ and $H^2(X,\sO_X)=0$ (see Lemma~\ref{lem:H2algXA} and Lemma~\ref{lem:H2algAAi}).
Even when $H^2(X,\sO_X)\neq 0$,
it often happens that the surfaces~$A_i$ are all birationally ruled and therefore satisfy the Tate conjecture.
This phenomenon is illustrated by the~$K3$ surfaces of Example~\ref{ex:sato} below.
We refer the reader to~\textsection\ref{sec:counterexample}
for an example of a simply connected surface~$X$ with positive geometric genus, over a $p$\nobreakdash-adic field,
which satisfies the assumptions of Theorem~\ref{th:kfinite} even though the~$A_i$'s are not birationally
ruled
(see Remark~\ref{rk:nonuniruledai}).

We note that Theorem~\ref{th:kfinite} should be applicable to any surface with $H^1(X,\sO_X)=0$;
for instance, to any~$K3$ surface.
As far as we are aware,
Theorem~\ref{th:kfinite} is the first general result about the kernel of
the cycle class map for $0$\nobreakdash-cycles on surfaces defined over a
$p$\nobreakdash-adic field which goes beyond the assumption $H^2(X,\sO_X)=0$
of Saito's theorem.
See~\cite{yamazaki} and the references therein for a discussion of this point and for a positive
result in the case of surfaces of the shape $C \times C'$ where~$C$ and~$C'$ are Mumford curves
over a $p$\nobreakdash-adic field.

In the proof of Theorem~\ref{th:kfinite} given below,
the Tate conjecture is used in Lemma~\ref{lem:tatech1}
and in the application of Lemma~\ref{lem:injh1h2},
while the hypothesis that the Albanese variety of~$X$ should have potentially good reduction
plays a role only in Lemma~\ref{lem:albgoodred}.  It would be interesting to understand 
the kernel of the cycle class map for $0$\nobreakdash-cycles on surfaces defined
over a $p$\nobreakdash-adic field when the Albanese variety of~$X$ does not have potentially good reduction.
See~\cite{satoinjectivity} for some work in this direction under the assumption that $H^2(X,\sO_X)=0$.

\begin{rmk}
\label{rk:leftkernel}
When~$X$ is a proper variety over a $p$\nobreakdash-adic field,
the natural pairing
\begin{align}
\label{eq:pairingmanin}
\CH_0(X) \times \Br(X) \to \Q/\Z
\end{align}
defined by Manin~\cite{manin} has been used as a tool to study the Chow group of $0$\nobreakdash-cycles of~$X$
(see \cite{ctsaito}, \cite{saitosatopadic}).
Theorem~\ref{th:kfinite} may be reinterpreted as asserting that for the surfaces~$X$ which appear in its statement,
the left kernel of~\eqref{eq:pairingmanin}
is divisible by~$\ell$.
(See,~\emph{e.g.}, \cite[p.~41]{ctcime} for this reformulation.)
This answers, for the prime-to\nobreakdash-$p$ part and conditionally on the Tate conjecture, the question
raised in \cite[p.~85]{parimalasuresh}, under the assumption that the Albanese variety of~$X$ has potentially good reduction.
\end{rmk}

\begin{example}
\label{ex:sato}
The~$K3$ surfaces over~$\Q_p$ considered by Sato in~\cite{satonondivisible}
possess a regular model over~$\Z_p$ whose special fiber is a union of smooth rational surfaces, with normal crossings
(see \cite[Example~3.7 and Corollary~A.7]{satonondivisible}).
As a consequence, they satisfy the assumptions of Theorem~\ref{th:kfinite}.
In particular, for these~$K3$ surfaces, the left kernel of the natural pairing
$\CH_0(X) \times \Br(X) \to \Q/\Z$
is divisible by~$\ell$ for any $\ell \neq p$.
\end{example}

\subsection{Local invariant cycle theorem}

In both~\textsection\ref{sec:surfacesoverpadic} and~\textsection\ref{sec:surfacesoverC((t))} we shall
apply some of the arguments of \cite[(3.6)]{deligneweil2} in an unequal characteristic situation.
For ease of reference we spell out the relevant fact in Lemma~\ref{lem:invcycleeq} below.
Let us take up the notation of~\textsection\ref{sec:criterion}
and assume~$k$ is separably closed.
Recall that for $i \geq 0$, one defines
a specialization map
\begin{align}
\label{eq:spmap}
H^i(A, \Q_\ell)\longrightarrow H^0(K,H^i(\bar X, \Q_\ell))
\end{align}
by composing the inverse of the proper base change isomorphism
$H^i(\sX,\Q_\ell) \isoto H^i(A,\Q_\ell)$
with the natural maps $H^i(\sX,\Q_\ell)\to H^i(X,\Q_\ell) \to H^0(K,H^i(\bar X,\Q_\ell))$
(see \cite[p.~256]{sga412}).
The term ``local invariant cycle theorem'' refers to the conjecture
that the map~\eqref{eq:spmap}
is surjective for any $i \geq 0$.
This conjecture holds in equal characteristic
(see \cite[Theorem~5.12]{steenbrink}, \cite{gna}, \cite[Th\'eor\`eme~3.6.1]{deligneweil2}, \cite{ito}, \cite[\textsection3.9]{illusiemonodromie}).  It~is also known to hold in mixed characteristic,
when~$k$ is an algebraic closure of a finite field
and either $d \leq 2$ or $i\leq 2$ (see \cite[Satz~2.13]{rapoportzink}, \cite[Lemma~3.9]{saitoweight}).
The~local invariant cycle theorem for any~$i$ and any separably closed~$k$ would follow from the monodromy-weight conjecture
(see \cite[Corollaire~3.11]{illusiemonodromie}).
We note that even though some of the above references assume that~$A$ is a simple normal crossings divisor in~$\sX$,
the general case reduces to this one thanks to de~Jong's theorem on the existence of alterations~\cite{dejong} and a trace argument.

\begin{lem}
\label{lem:invcycleeq}
Let $i \geq 0$. Under the assumptions of Theorem~\ref{th:criterionsepclosed},
the specialization map~\eqref{eq:spmap}
is surjective
if and only if the sequence
\begin{align}
\label{eq:seqinvcycle}
\xymatrix{
H^m_A(\sX,\Ql) \ar[r] & H^m(\sX,\Ql) \ar[r] & H^0(K,H^m(\bar X,\Ql))
}
\end{align}
is exact for $m=2d+1-i$.
\end{lem}

\begin{proof}
The localization exact sequence for \'etale cohomology and the Hochschild--Serre spectral sequence
give rise to a commutative diagram
\begin{align*}
\xymatrix@C=1.5em{
&& H^0(K,H^m(\bar X,\Ql)) \\
H^m_A(\sX,\Ql) \ar[r] & H^m(\sX,\Ql) \ar[r]\ar[ur] & H^m(X,\Ql) \ar[u]\ar[r] & H^{m+1}_A(\sX,\Ql) \\
&& H^1(K,H^{m-1}(\bar X,\Ql)) \ar@{_{(}->}[u] \ar[ur]
}
\end{align*}
with exact row and column.
The bottom slanted arrow is dual, up to a twist, to the specialization map~\eqref{eq:spmap}
(see \textsection\ref{subsec:duality}
and Remark~\ref{rk:pairingscompatible}),
therefore it is injective if and only if~\eqref{eq:spmap} is surjective. The lemma follows by a diagram chase.
\end{proof}

\subsection{From rational to divisible torsion coefficients}

The following lemma will play a crucial role both in the proof of Theorem~\ref{th:kfinite} and in~\textsection\ref{sec:surfacesoverC((t))}.
It fails for varieties of dimension~$\geq 3$, both in the case of separably closed
fields and in that of finite fields.
This is the main reason why Theorem~\ref{th:kfinite} requires the assumption $d=2$.

\begin{lem}
\label{lem:QtoQ/Z}
Let~$V$ be a proper surface with simple normal crossings over a field~$k$.
Assume~$k$ is either finite or separably closed.
Denote by $(V_i)_{i \in I}$ the family of irreducible components of~$V$,
and let $m=2$ if~$k$ is separably closed, $m=3$ if~$k$ is finite.
Let~$\ell$ be a prime number invertible in~$k$.
If the restriction map
\begin{align}
\label{eq:restmapQtoQ/Z}
H^m(V, \Ql(1))\to\bigoplus_{i\in I} H^m(V_i, \Ql(1))
\end{align}
is injective,
then so is the restriction map
\begin{align}
H^m(V, \Ql/\Zl(1))\to\bigoplus_{i\in I} H^m(V_i, \Ql/\Zl(1))\rlap{\text{.}}
\end{align}
\end{lem}

\begin{proof}
Let us fix a total ordering on the finite set~$I$.  For $n \geq 0$, set
\begin{align*}
V^{(n)}=\coprod_{i_1 < \dots < i_{n+1}} \big(V_{i_1} \cap \dots \cap V_{i_{n+1}}\big)
\end{align*}
and let $a_n:V^{(n)}\to V$ denote the natural finite map.  For $n=0$, this is the normalization map.
The maps appearing in the statement of Lemma~\ref{lem:QtoQ/Z} fit into the following commutative diagram, whose rows are exact:
\begin{align}
\label{diag:qqz}
\begin{aligned}
\xymatrix{
\displaystyle H^{m-1}\bigg(V,\frac{a_{0*}\Ql(1)}{\Ql(1)}\bigg) \ar[d] \ar[r] & H^m(V,\Ql(1)) \ar[d] \ar[r] & H^m(V,a_{0*}\Ql(1)) \ar[d] \\
\displaystyle H^{m-1}\bigg(V,\frac{a_{0*}(\Ql/\Zl)(1)}{\Ql/\Zl(1)}\bigg) \ar[r] & H^m(V,\Ql/\Zl(1)) \ar[r] & H^m(V,a_{0*}\Ql/\Zl(1))\rlap{\text{.}} \\
}
\end{aligned}
\end{align}
In order to establish Lemma~\ref{lem:QtoQ/Z}, it suffices to prove that the leftmost vertical arrow of this diagram is onto.

As~$V$ is a surface with simple normal crossings, we have $\dim\mkern-2mu\big(V^{(n)}\big)\leq 2-n$ for any~$n$,
and
the choice of the ordering on~$I$ determines a natural exact sequence of $\Zl$\nobreakdash-sheaves
\begin{align}
\label{eq:a0a1a2}
\xymatrix{
0 \ar[r] & \Zl(1) \ar[r] & a_{0*} \Zl(1) \ar[r] & a_{1*} \Zl(1) \ar[r] & a_{2*} \Zl(1) \ar[r] & 0\rlap{\text{.}}
}
\end{align}
For any irreducible curve~$C$ proper over~$k$, we have $H^m(C,\Zl(1))=\Zl$.
Therefore the group
$H^m(V,a_{1*}\Zl(1))=\bigoplus_{i<j} H^m(V_i \cap V_j,\Zl(1))$ is torsion-free.
Similarly, we have $H^{m-1}(V,a_{2*}\Zl(1))=0$ for reasons of cohomological dimension.
As~$k$ is finite or separably closed, short exact sequences of
$\Zl$\nobreakdash-sheaves give rise to long exact sequences of cohomology groups;
hence from the previous remarks and from~\eqref{eq:a0a1a2} we deduce that
the group $H^m\bigg(V,\displaystyle\frac{a_{0*}\Zl(1)}{\Zl(1)}\bigg)$ is torsion-free.
In view of the exact sequence
$$
\xymatrix{
0 \ar[r] & \displaystyle\frac{a_{0*}\Zl(1)}{\Zl(1)} \ar[r] &
\displaystyle \frac{a_{0*}\Ql(1)}{\Ql(1)} \ar[r] & 
\displaystyle \frac{a_{0*}(\Ql/\Zl)(1)}{\Ql/\Zl(1)} \ar[r] & 0\rlap{\text{,}}
}
$$
we conclude that the leftmost vertical arrow of~\eqref{diag:qqz} is indeed onto.
\end{proof}

\subsection{Two lemmas on cohomology with rational coefficients}

In the next two lemmas, we study the kernel of~\eqref{eq:restmapQtoQ/Z}
when~$k$ is a finite field.  Following Tate~\cite[p.~72]{tateseattle}, we say that an endomorphism~$\phi$
of a finite-dimensional vector space is \emph{partially semisimple} if the map $\Ker(\phi-1) \to \Coker(\phi-1)$ induced by the identity
is bijective (equivalently, if~$1$ is not a multiple root of the minimal polynomial of~$\phi$).

\begin{lem}
\label{lem:injh1h2}
Let~$V$ be a proper variety with simple normal crossings over a finite field~$k$.
Denote by $(V_i)_{i \in I}$ the family of irreducible components of~$V$.
Let~$\ell$ be a prime number invertible in~$k$.
If the Frobenius endomorphism of $H^2(\bar V_i,\Ql(1))$ is partially semisimple for all~$i \in I$,
the restriction map
$$
H^1(k,H^2(\bar V,\Ql(1))) \longrightarrow \bigoplus_{i \in I} H^1(k,H^2(\bar V_i,\Ql(1)))
$$
is injective.
\end{lem}

\begin{proof}
As~$k$ is finite, we may consider the weights of Frobenius acting on $H^2(\bar V,\Ql(1))$ and on
$M=\bigoplus_{i \in I} H^2(\bar V_i,\Ql(1))$.
Let $r:H^2(\bar V,\Ql(1)) \to M$ denote the restriction map.
According to the Mayer--Vietoris spectral sequence associated with the finite covering $V=\bigcup_{i \in I} V_i$
and to Deligne's theorem
\cite[Th\'eor\`eme~1]{deligneweil2},
the weights of $\Ker(r)$ are~$\leq -1$.
Therefore $H^1(k,\Ker(r))=0$ and hence the map $H^1(k,H^2(\bar V,\Ql(1))) \to H^1(k,\Im(r))$ induced by~$r$ is injective.
On the other hand, since the Frobenius endomorphism acts partially semisimply on~$M$,
it also acts partially semisimply on $\Im(r)$. We may thus identify the groups $H^i(k,\Im(r))$ and $H^i(k,M)$ for $i=0$ with the same groups for $i=1$.
As $H^0(k,\Im(r))$ injects into $H^0(k,M)$, it follows that
$H^1(k,\Im(r))$ injects into $H^1(k,M)$.
\end{proof}

\begin{lem}
\label{lem:albgoodred}
With the notation of \textsection\ref{sec:criterion}, assume that $d=2$,
that~$k$ is finite, and that
the Albanese variety of~$X$ has potentially good reduction.   Then the restriction map
$H^3(\bar A,\Ql(1)) \to \bigoplus_{i\in I} H^3(\bar A_i,\Ql(1))$ is injective.
\end{lem}

\begin{proof}
Let $R^\nr$ denote the strict henselization of~$R$, let $K^\nr$ denote its quotient field and let $\sX^\nr = \sX \otimes_R R^\nr$.
The local invariant cycle theorem holds for $H^2(\bar X,\Ql)$, according to Rapoport and Zink \cite[Satz~2.13]{rapoportzink}.
As a consequence, by Lemma~\ref{lem:invcycleeq}, we have an exact sequence of finite-dimensional vector spaces endowed with an action of Frobenius
\begin{align}
\label{se:weightsh3}
\xymatrix{
H^3_{\bar A}(\sX^\nr,\Ql(1)) \ar[r] & H^3(\bar A,\Ql(1)) \ar[r] & H^0(K^\nr,H^3(\bar X,\Ql(1)))\rlap{\text{.}}
}
\end{align}
Let~$T$ denote the $\ell$\nobreakdash-adic Tate module, with rational coefficients,
of the Albanese variety of~$X$.
Let $L^\nr/K^\nr$ be a finite field extension over which this abelian variety acquires good reduction.
The vector space $H^3(\bar X,\Ql(2))$ is canonically isomorphic to~$T$, since both spaces are canonically dual to $H^1(\bar X,\Ql)$.
As the Albanese variety of~$X$ has good reduction over $L^\nr$, we know, by Weil, that the action of Frobenius on $H^0(L^\nr,T)$, and therefore also on the subspace $H^0(K^\nr,T)$,
is pure of weight~$-1$
(see \cite[Exp.~I, \textsection6.4]{sga71}).
It~follows that $H^0(K^\nr,H^3(\bar X,\Ql(1)))$ is pure of weight~$1$.
On the other hand, the vector space $H^3_{\bar A}(\sX^\nr,\Ql(1))$ is dual to $H^3(\bar A,\Ql(2))$ (see Proposition~\ref{prop:dualityAsX}),
so that its weights are positive, by~\cite[Th\'eor\`eme~1]{deligneweil2}.
We conclude, thanks to~\eqref{se:weightsh3}, that $H^3(\bar A,\Ql(1))$ has positive weights (and is in fact pure of weight~$1$).
Now,
as in the proof of Lemma~\ref{lem:injh1h2},
the Mayer--Vietoris spectral sequence and~\cite{deligneweil2}
imply together that the kernel of the restriction map $H^3(\bar A,\Ql(1)) \to \bigoplus_{i\in I} H^3(\bar A_i,\Ql(1))$ has weights~$\leq 0$.
Hence the lemma.
\end{proof}

\subsection{Proof of Theorem~\ref{th:kfinite}}

Let us assume the hypotheses of Theorem~\ref{th:kfinite} are satisfied.
For $i \in I$, the Frobenius endomorphism of $H^2(\bar A_i,\Ql(1))$ is
partially semisimple
since the Tate conjecture holds for the surface~$A_i$
(see \cite[Proposition~2.6]{tateseattle}, which may be applied thanks to
\cite[Th\'eor\`eme~4.6]{kleimanfinitude}).  It thus follows from Lemma~\ref{lem:injh1h2}, Lemma~\ref{lem:albgoodred}, and from the Hochschild--Serre spectral sequence,
that the restriction map
\begin{align}
\label{eq:restmap}
H^3(A,\Ql(1)) \longrightarrow \bigoplus_{i \in I} H^3(A_i,\Ql(1))
\end{align}
is injective.

\begin{lem}
\label{lem:tatech1}
For all $i \in I$, the subgroup of $H^3(A_i,\Ql/\Zl(1))$ denoted $\CH_1(A_i)^\perp$ in~\eqref{eq:complexfinite} vanishes.
\end{lem}

\begin{proof}
As~$A_i$ is a surface over a finite field satisfying the Tate conjecture, the cycle class map
$\CH_1(A_i) \otimeshat \Zl \to H^2(A_i,\Zl(1))$ is surjective
(see \cite[Proposition~4.3]{tateseattle}).
On the other hand, the cup-product pairing $H^3(A_i,\Ql/\Zl(1)) \times H^2(A_i,\Zl(1)) \to \Ql/\Zl$
is perfect (see~\textsection\ref{subsec:duality}). The lemma follows.
\end{proof}

Putting together Lemma~\ref{lem:QtoQ/Z}, Lemma~\ref{lem:tatech1} and the injectivity of~\eqref{eq:restmap},
we deduce that
the natural map
\begin{align}
\label{eq:injh3proof}
H^3(A,\Ql/\Zl(1)) \longrightarrow \bigoplus_{i\in I} \frac{H^3(A_i,\Ql/\Zl(1))}{\CH_1(A_i)^\perp}
\end{align}
is injective.  In view of the localization exact sequence
\begin{align}
\xymatrix{
H^2(X, \Ql/\Zl(1)) \ar[r] & H^3_A(\sX,\Ql/\Zl(1)) \ar[r] & H^3(A,\Ql/\Zl(1))\rlap{\text{,}}
}
\end{align}
the injectivity of~\eqref{eq:injh3proof} implies, in turn, the exactness of the complex~\eqref{eq:complexfinite}.
Applying Theorem~\ref{th:criterionfinite} then concludes the proof.

\section{Surfaces with \texorpdfstring{$H^2(X,\sO_X)=0$}{H²(X,O)=0} over~\texorpdfstring{$K$}{K} with \texorpdfstring{$K=\C((t))$ or $[K:\Q_p^\nr]<\infty$}{K=ℂ((t)) or [K:ℚₚⁿʳ]<∞}}
\label{sec:surfacesoverC((t))}

If~$V$ is a finite type scheme over a separably closed or finite field,
we say that $H^2(V,\Ql(1))$ is \emph{algebraic} if the natural injection $\Pic(V) \otimeshat \Ql \hookrightarrow H^2(V,\Ql(1))$ is an isomorphism.
Recall that when~$V$ is smooth and projective (or proper) over an algebraically closed field of characteristic~$0$,
this condition is equivalent to $H^2(V,\sO_V)=0$, by Hodge theory.

In a series of papers originating with the work of Bloch~\cite{blochtorsion},
K\nobreakdash-theoretic methods were applied to the study of torsion cycles of codimension~$2$ on smooth projective varieties defined over various types of fields.
By combining a result stemming from this series,
namely, a theorem of Colliot-Th\'el\`ene and Raskind~\cite[Theorem~3.13]{ctraskind},
with the finiteness theorem of Saito and Sato \cite[Theorem~9.7]{saitosato}
and with Roitman's theorem,
it is not hard to deduce\footnote{Strictly speaking, the reference \cite[Theorem~3.13]{ctraskind} can only be applied
when~$K$ has characteristic~$0$ and $H^1(\bar X,\Q/\Z)=0$.  It is nevertheless possible to show, even though we do not do it here, that variants of the arguments contained in \emph{loc. cit.}
are sufficient
to deduce Theorem~\ref{th:H2alg-ksepclosed} in full generality by K\nobreakdash-theoretic methods, the main point being that
the finiteness result \cite[Theorem~9.7]{saitosato} implies the surjectivity
of the natural map from the $\ell$\nobreakdash-primary torsion subgroup of~$A_0(X)$
to $A_0(X)\otimeshat\Zl$.} the following statement.

\begin{thm}
\label{th:H2alg-ksepclosed}
With the notation of \textsection\ref{sec:criterion}, assume
that $d=2$, that~$k$ is separably closed,
and that the reduced special fiber $A_\red$ of~$\sX$ has simple normal crossings.
If $H^2(\bar X,\Ql(1))$ is algebraic, then the cycle class map $\CH_0(X) \otimeshat \Zl \to H^{2d}(X, \Zl(d))$ is injective.
If in addition $H^1(\bar X,\Z/\ell\Z)=0$, then
the group $A_0(X)$ is divisible by~$\ell$.
\end{thm}

In this section, we give a new proof of Theorem~\ref{th:H2alg-ksepclosed},
based on the criterion established in \textsection\ref{sec:criterion}.
The argument proceeds by an analysis of the special
fiber of the model~$\sX$ and is thus quite different from
the K\nobreakdash-theoretic approach alluded to above.
We hope that it may shed some light on the problem at hand, especially in
the higher-dimensional case, where codimension~$2$ cycles are no longer relevant for the study of zero-cycles.

The general strategy for the proof of Theorem~\ref{th:H2alg-ksepclosed} follows the same lines as that for the proof of Theorem~\ref{th:kfinite}.
Two new difficulties arise, however.  First, as is clear from the statements of Theorems~\ref{th:criterionsepclosed} and~\ref{th:criterionfinite},
the relevant restriction map is no longer $H^3(A,\Ql(1)) \to \bigoplus_{i \in I} H^3(A_i,\Ql(1))$,
where~$A$ is a surface with simple normal crossings over a finite field,
but $H^2(A,\Ql(1)) \to \bigoplus_{i \in I} H^2(A_i,\Ql(1))$, where~$A$ is now defined over a separably closed field.
In contrast with the situation of~\textsection\ref{sec:surfacesoverpadic},
the latter map need not be injective, even if the Albanese variety of~$X$ is trivial.
The second difficulty is that
the groups denoted $\CH_1(A_i)^\perp$ in~\eqref{eq:complexsepclosed},
contrary to the groups denoted~$\CH_1(A_i)^\perp$ in~\eqref{eq:complexfinite},
need not vanish.
This phenomenon is related to the defect of unimodularity of the intersection pairing on~$\NS(A_i)$.
We prove in~\textsection\ref{subsec:algspec} that both of these difficulties disappear when $H^2(\bar X,\Ql(1))$ is algebraic.
We refer the reader
to~\textsection\textsection\ref{sec:semistableK3}--\ref{sec:counterexample} for results in two situations in
which $H^2(\bar X,\Ql(1))$ has a nontrivial transcendental quotient.

\begin{example}
\label{ex:n=3}
Theorem~\ref{th:H2alg-ksepclosed} applies to the surface considered in Example~\ref{ex:isotrivialdegn}
when $n \leq 3$.  More generally, if~$X$ is a (geometrically) rational surface over $K=\C((t))$,
then $A_0(X)$ has finite exponent (being killed by the degree of any extension of~$K$ over which~$X$ becomes rational), and is divisible
according to Theorem~\ref{th:H2alg-ksepclosed}, so that $A_0(X)=0$.
Thus we recover in this case the conclusion of \cite[Theorem~A~(iv)]{ctk2}.
\end{example}

\begin{rmk}
\label{rk:unramifiedh3}
Theorem~\ref{th:H2alg-ksepclosed} may be reinterpreted as a statement about the unramified cohomology of~$X$.
Let us assume that its hypotheses are satisfied and, for simplicity,
that~$K$ has characteristic~$0$ and that $H^1(\bar X,\Z/\ell\Z)=0$.
The group $A_0(X)$ is then divisible by~$\ell$
if and only if the unramified cohomology group
$H^3_\nr(X,\Ql/\Zl(2))$ vanishes.  Indeed, according to \cite[Th\'eor\`eme~8.7]{ctvoisin}, we have an exact sequence
\begin{align*}
\xymatrix{
0 \ar[r] & A_0(X) \ar[r] & H^0(K,A_0(\bar X)) \ar[r] & H^3_\nr(X,\Q/\Z(2)) \ar[r] & 0
}
\end{align*}
(noting that $H^3_\nr(\bar X,\Q/\Z(2))=0$ as~$X$ is a surface),
and Roitman's theorem implies that~$A_0(\bar X)$, and therefore also $H^0(K,A_0(\bar X))$, is a $\Q$\nobreakdash-vector space.
\end{rmk}

\subsection{Algebraicity of \texorpdfstring{$H^2$}{H²}, specialization, and normalization}
\label{subsec:algspec}

We keep the notation introduced at the beginning of \textsection\ref{sec:criterion}, and assume henceforth that~$k$ is separably closed.
The goal of \textsection\ref{subsec:algspec} is to establish the following three lemmas.

\begin{lem}
\label{lem:H2algXA}
If $H^2(\bar X,\Ql(1))$ is algebraic, then $H^2(A,\Ql(1))$ is algebraic.
\end{lem}

\begin{lem}
\label{lem:H2algAAi}
Assume $d=2$ and~$A_i$ is smooth for every $i \in I$.
If $H^2(A,\Ql(1))$ is algebraic, 
then $H^2(A_i,\Ql(1))$ is algebraic for all $i\in I$.
\end{lem}

\begin{lem}
\label{lem:weights}
Let $\pi:W \to V$ be a surjective morphism between proper schemes over a separably closed field~$k$.
Let~$\ell$ be a prime number invertible in~$k$.
If $H^2(V,\Ql(1))$ is algebraic, the map $\pi^*:H^2(V,\Ql(1))\to H^2(W,\Ql(1))$ is injective.
\end{lem}

Lemma~\ref{lem:weights} will be applied to $V=A_\red$ and to the normalization morphism $\pi:W\to V$.

The statement of Lemma~\ref{lem:H2algAAi} fails when $d>2$, as shown by the example of~$\P^3_R$ blown up
along a smooth quartic surface contained in the special fiber (in this example $H^2(A,\Ql(1))$ is algebraic by Lemma~\ref{lem:H2algXA}).

We note that when~$k$ has characteristic~$0$, Lemmas~\ref{lem:H2algXA} and~\ref{lem:H2algAAi} may be rephrased in terms of Hodge theory
and are then straightforward consequences of
Koll\'ar's torsion-freeness theorem~\cite{kollartorsionfreeness}
(see also~\cite{duboisjarraud}, \cite[Theorem~2.11]{steenbrinkoslo})
or of
the Clemens--Schmid exact sequence (see~\cite[Theorem~2.7.5]{persson}).
Similarly, when~$k$ has characteristic~$0$, the reduction to finite fields which appears in the proof of Lemma~\ref{lem:weights}
could be replaced with a reduction to $k=\C$ and the use of Hodge theory.

For lack of an adequate reference, we start with a proof of the local invariant cycle theorem for $H^{2d-1}$ over an arbitrary strictly henselian excellent discrete valuation ring.

\begin{lem}
\label{lem:locinvcycle}
The specialization map $H^{2d-1}(A,\Ql) \to H^0(K,H^{2d-1}(\bar X,\Ql))$ is surjective.
\end{lem}

\begin{proof}
For simplicity we assume that $A_\red$ has simple normal crossings.  The general case follows thanks to de~Jong's theorem~\cite{dejong}
and a trace argument.
Fix an embedding $\sX \hookrightarrow \P^N_R$ for some~$N$
and let~$\sY$ denote the intersection of~$\sX$ with $d-1$
degree~$M$ hypersurfaces of~$\P^N_R$ which lift general hypersurfaces of~$\P^N_k$.
By the Bertini theorem, if $M\gg 0$, the scheme~$\sY$ is irreducible, regular, flat over~$R$, of dimension~$2$
(see Lemma~\ref{app:lem_curve}).
Let $B = \sY \otimes_R k$ and $\bar Y = \sY \otimes_R \bar K$.
The specialization maps fit into a commutative diagram
\begin{align}
\begin{aligned}
\xymatrix@R=2em{
H^1(A, \Ql(1)) \ar[d] \ar[r] & H^0(K, H^1(\bar X, \Ql(1))) \ar[d] \\
H^1(B, \Ql(1)) \ar[d] \ar[r] & H^0(K, H^1(\bar Y, \Ql(1))) \ar[d] \\
H^{2d-1}(A, \Ql(d))   \ar[r] & H^0(K, H^{2d-1}(\bar X, \Ql(d)))
}
\end{aligned}
\end{align}
where the vertical arrows are the restriction maps and the Gysin maps associated with the embedding $\sY \hookrightarrow \sX$.
By the Hard Lefschetz theorem, the composition of the two right-hand side vertical arrows is an isomorphism.
It follows that the surjectivity of the middle horizontal arrow implies that of the bottom horizontal arrow.  In other words, we may replace~$A$ and~$\bar X$ with~$B$ and~$\bar Y$ and
thus assume $d=1$.
In this case, the complex~\eqref{eq:seqinvcycle} for $m=2$ identifies, up to a twist, with the complex
\begin{align*}
\xymatrix{
\displaystyle\bigoplus_{i\in I} \Ql \ar[r] & \displaystyle\bigoplus_{i\in I} \Ql \ar[r] & \Ql
}
\end{align*}
whose first arrow is given by the intersection matrix of the reduced special fiber~$A_\red$, and whose second arrow maps the $i$th
basis vector to the multiplicity of~$A_i$ in~$A$.
This complex is well known to be exact (see, \emph{e.g.}, \cite[Chapter~III, Remark~8.2.3]{silvermanadvanced}).
Applying Lemma~\ref{lem:invcycleeq} now concludes the proof.
\end{proof}

\begin{proof}[Proof of Lemma~\ref{lem:H2algXA}]
By Lemma~\ref{lem:locinvcycle} and Lemma~\ref{lem:invcycleeq}, we have an exact sequence
\begin{align}
\label{seq:h2}
\xymatrix{
H^2_A(\sX,\Ql(1)) \ar[r] & H^2(\sX,\Ql(1)) \ar[r] & H^0(K,H^2(\bar X,\Ql(1)))\rlap{\text{.}}
}
\end{align}
We have $H^2_A(\sX,\Ql(1)) = \bigoplus_{i \in I} \Ql$ (see Remark~\ref{rk:complexsepclosed}~(i)) and the first arrow of~\eqref{seq:h2} maps the $i$th basis vector to $c(\sO_\sX(A_i))$,
where~$c$ denotes the natural map
\begin{align*}
c:\Pic(\sX) \otimeshat \Ql \to H^2(\sX,\Ql(1))\rlap{\text{.}}
\end{align*}
Thus $c(\Pic(\sX)\otimeshat\Ql)$ contains the image of $H^2_A(\sX,\Ql(1))$.
On the other hand, as $H^2(\bar X,\Ql(1))$ is algebraic and~$\sX$ is regular, Lemma~\ref{lem:piconto} below implies that
$c(\Pic(\sX)\otimeshat\Ql)$ surjects onto $H^0(K,H^2(\bar X,\Ql(1)))$.
Hence~$c$ is surjective.
By the proper base change theorem, we conclude
that $H^2(A,\Ql(1))$ is algebraic.
\end{proof}

\begin{lem}
\label{lem:piconto}
For any variety~$X$ smooth and proper over a field~$K$,
the natural map $\Pic(X) \otimeshat \Ql \to H^0(K,\Pic(\bar X) \otimeshat \Ql)$ is onto.
\end{lem}

\begin{proof}
We have $\Pic(\bar X)\otimeshat\Ql=\NS(\bar X)\otimeshat \Ql=\NS(\bar X)\otimes_\Z \Ql$
as $\Pic^0(\bar X)$ is divisible and $\NS(\bar X)$ is finitely generated.
It follows that the natural map $\Div(\bar X) \otimes_\Z \Ql \to \Pic(\bar X)\otimeshat\Ql$ is surjective.
On the other hand, any surjective map between discrete Galois modules which are $\Q$\nobreakdash-vector spaces induces
a surjection on Galois invariants.
Therefore $\Div(X) \otimes_\Z \Ql$, and thus also $\Pic(X) \otimes_\Z \Ql$,
surjects onto $H^0(K,\Pic(\bar X)\otimeshat\Ql)$.
\end{proof}

\begin{proof}[Proof of Lemma~\ref{lem:H2algAAi}]
Fix $i \in I$ and set $E=A_i$ and $F=\bigcup_{j \in I \setminus \{i\}} A_j$.
Note that~$E$ is a smooth and proper surface.
Let $\alpha \in H^2(E,\Ql(1))$ be orthogonal to $\NS(E) \otimes_\Z \Ql$ with respect to the cup-product pairing.
As~$\alpha$ is in particular orthogonal to the classes of the irreducible curves contained in~$E \cap F$, the
Mayer--Vietoris exact sequence
\begin{align}
\xymatrix{
H^2(A,\Ql(1)) \ar[r] & H^2(E,\Ql(1)) \oplus H^2(F,\Ql(1)) \ar[r] & H^2(E \cap F,\Ql(1))
}
\end{align}
shows that~$\alpha$ is the restriction of an element of $H^2(A,\Ql(1))$.
Assuming that $H^2(A,\Ql(1))$ is algebraic, it follows that $\alpha \in \NS(E) \otimes_\Z \Ql$.
On the other hand, the pairing on $\NS(E) \otimes_\Z \Ql$ induced by cup-product is non-degenerate (see \cite[Th\'eor\`eme~4.6]{kleimanfinitude}).
Hence~$\alpha=0$.  We have thus shown that the orthogonal subspace to $\NS(E) \otimes_\Z \Ql$ in $H^2(E,\Ql(1))$ is trivial,
in other words $H^2(E,\Ql(1))$ is algebraic.
\end{proof}

\begin{proof}[Proof of Lemma~\ref{lem:weights}]
Let us choose a finitely generated $\Z[1/\ell]$\nobreakdash-subalgebra $R_0 \subset k$ such that $\pi:W \to V$ descends to a surjective morphism $\pi_0:W_0\to V_0$
between proper $R_0$\nobreakdash-schemes.
Let $f:V_0 \to \Spec(R_0)$ denote
the structure morphism of~$V_0$
and let~$C$ be the cone of the natural morphism $\Ql(1) \to R\pi_{0*}\Ql(1)$.
As the group $\NS(V)$ is finitely generated (see \cite[Th\'eor\`eme~5.1]{kleimanfinitude}),
we may assume,
after enlarging~$R_0$, that every element of $\NS(V)$ is represented by the inverse image, on~$V$, of a line bundle on~$V_0$.
Enlarging~$R_0$ again allows us to assume, in addition,
that the constructible $\Ql$\nobreakdash-sheaves
appearing in the exact sequence
\begin{align}
\label{se:lissesheaves}
\xymatrix{
R^1f_*C \ar[r] & R^2f_*\Ql(1) \ar[r] & R^2f_*\left(R\pi_{0*}\Ql(1)\right)
}
\end{align}
are lisse (see \cite[Chapter~VI, Theorem~2.1]{milne}, \cite[Chapter~I, \textsection12, Proposition~12.10]{freitagkiehl}).

According to \cite[Th\'eor\`eme~1]{deligneweil2},
for any $i\geq 0$,
the $i$th cohomology sheaf of~$C$ is mixed of weights~$\leq i-2$.
Moreover, as~$\pi_0$ is surjective, the complex~$C$ is concentrated in non-negative degrees.
Thus~$C$ is a mixed complex of weights~$\leq -2$
in the sense of \cite[D\'efinition~6.2.2]{deligneweil2}.
It~follows that the
$\Ql$\nobreakdash-sheaf $R^1f_*C$ is mixed of negative weights
(see \emph{loc.\ cit.}, Variante~6.2.3).
On the other hand, our hypotheses imply that the global sections of the $\Ql$\nobreakdash-sheaf $R^2f_*\Ql(1)$ generate its geometric generic stalk.
Being a lisse sheaf, it must then be constant, and therefore, punctually pure of weight~$0$.
As a result, the first map of~\eqref{se:lissesheaves} must vanish, and the second map is an injection.
In particular it induces an injection on the geometric generic stalks.
\end{proof}

\subsection{Proof of Theorem \ref{th:H2alg-ksepclosed}}

We now assume the hypotheses of Theorem~\ref{th:H2alg-ksepclosed} are satisfied.
Lemma~\ref{lem:H2algXA} and Lemma~\ref{lem:H2algAAi} imply that $H^2(A_i,\Ql(1))$ is algebraic for all $i \in I$.
In particular, we have $\CH_1(A_i) \otimeshat \Zl = \Pic(A_i) \otimeshat \Zl = H^2(A_i,\Zl(1))$ for any~$i$, since the cokernel of
the natural injection $\Pic(A_i) \otimeshat \Zl \hookrightarrow H^2(A_i,\Zl(1))$ is torsion-free.
It follows that the group denoted $\CH_1(A_i)^\perp$ in~\eqref{eq:complexsepclosed} vanishes for all~$i\in I$,
as the cup-product pairing
\begin{align*}
H^2(A_i,\Ql/\Zl(1)) \times H^2(A_i,\Zl(1)) \to \Ql/\Zl
\end{align*}
is a perfect pairing.
On the other hand, 
Lemma~\ref{lem:H2algXA} enables us to apply
Lemma~\ref{lem:weights} and 
Lemma~\ref{lem:QtoQ/Z} to $V=A$ and $W=\coprod_{i \in I} A_i$.
From all of this, we deduce
that
the natural map
\begin{align}
H^2(A,\Ql/\Zl(1)) \longrightarrow \bigoplus_{i\in I} \frac{H^2(A_i,\Ql/\Zl(1))}{\CH_1(A_i)^\perp}
\end{align}
is injective, and hence, by the localization exact sequence, that the complex~\eqref{eq:complexsepclosed} is exact.
Theorem~\ref{th:H2alg-ksepclosed} now results from Theorem~\ref{th:criterionsepclosed} and from Remark~\ref{rk:complexsepclosed}~(iii).

\subsection{A remark on the cycle class map for \texorpdfstring{$1$-cycles}{1-cycles} on the special fiber}
\label{subsec:aremarkpsi1A}

In~the proofs of Theorem~\ref{th:kfinite} and Theorem~\ref{th:H2alg-ksepclosed}, we
have shown the exactness of the complexes~\eqref{eq:complexsepclosed}
and~\eqref{eq:complexfinite} by establishing a stronger property, namely
the injectivity of the natural map
\begin{align}
\label{eq:injdualpsi}
H^{m}(A,\Ql/\Zl(1)) \longrightarrow \bigoplus_{i\in I} \frac{H^{m}(A_i,\Ql/\Zl(1))}{\CH_1(A_i)^\perp}\rlap{\text{,}}
\end{align}
where $m=2$ if~$k$ is separably closed and $m=3$ if~$k$ is finite.
By arguments similar to those of~\textsection\ref{subsec:proofcriterion},
the injectivity of~\eqref{eq:injdualpsi} is equivalent to the
surjectivity of the
$\ell$\nobreakdash-adic cycle class
map
\begin{align}
\label{eq:mappsi}
\psi_{1,A}:\CH_1(A)\otimeshat\Zl \to H^{2d}_A(\sX,\Zl(d))
\end{align}
for $1$\nobreakdash-cycles on~$A$.

We remark that outside the scope of Theorems~\ref{th:kfinite} and~\ref{th:H2alg-ksepclosed},
these two equivalent conditions very often fail.  A trivial example is given
by any surface over~$\C((t))$ with positive geometric genus and good reduction.
More generally,
analysing
the group $H^{2d}_A(\sX,\Ql(d))$ 
with the help of the Mayer--Vietoris spectral sequence
yields
the following lemma.

\begin{lem}
\label{lem:generallemmapsi}
With the notation of \textsection\ref{sec:criterion}, assume
that $d=2$, that~$k$ is separably closed,
and that the reduced special fiber $A_\red$ of~$\sX$ has simple normal crossings.
If the map~\eqref{eq:mappsi} is surjective, then $H^2(A,\Ql(1))$ is algebraic.
In particular, if~$k=\C$ and at least one of the surfaces~$A_i$ has positive geometric genus, then~\eqref{eq:mappsi} cannot be surjective.
\end{lem}

This lemma applies in particular to the situation considered in
Example~\ref{ex:isotrivialK3}, in which we have shown the injectivity of the cycle class map $\psi_{0,X}:\CH_0(X)\otimeshat\Zl\to H^4(X,\Zl(2))$ for a certain isotrivial~$K3$ surface~$X$ over~$\C((t))$:
the map~$\psi_{1,A}$ fails to be surjective in this example.
A larger source of examples of this kind is given by~$K3$ surfaces with semistable reduction over~$\C((t))$.
For such a surface, if~$A$ denotes the special fiber of a semistable model,
we have $H^2(A,\sO_A)\neq 0$
by~\cite{duboisjarraud}, so that $H^2(A,\Ql(1))$ is not algebraic.
Thus,
by Lemma~\ref{lem:generallemmapsi},
the map~$\psi_{1,A}$ is \emph{never} surjective in the case of a semistable~$K3$ surface
over~$\C((t))$.
We~shall prove in~\textsection\ref{sec:semistableK3} that~$\psi_{0,X}$ is
nonetheless injective for these surfaces.

\section{Semistable \texorpdfstring{$K3$}{K3} surfaces}
\label{sec:semistableK3}

As a first step towards the study of surfaces with nonzero geometric
genus over strictly local fields, we consider, in this section, semistable $K3$ surfaces, and prove:

\begin{thm}
\label{th:semiK3}
For any semistable $K3$ surface~$X$ over $\C((t))$, the group $A_0(X)$ is divisible.
\end{thm}

The proof is based on the criterion established in~\textsection\ref{sec:criterion}
and on the results of Kulikov, Persson, Pinkham,
and of Miranda and Morrison about degenerations of analytic~$K3$ surfaces over the punctured
unit disk (see \cite{kulikov}, \cite{perssonpinkham},
\cite{mirandamorrison}).  It is only in order to use these results that we need the semistability assumption.
Theorem~\ref{th:semiK3} and Example~\ref{ex:isotrivialK3} suggest that $A_0(X)$ may in fact be divisible for all~$K3$ surfaces~$X$ over~$\C((t))$.
Similarly, one might hope that $A_0(X)$ is divisible by~$\ell$ for any~$K3$ surface~$X$ over the maximal unramified extension of a $p$\nobreakdash-adic field,
and any $\ell \neq p$.

Before we start the proof of Theorem~\ref{th:semiK3}, let us recall the results we shall need on degenerations of analytic~$K3$ surfaces,
and set up the relevant terminology.
We say that a variety over~$\C((t))$ is \emph{semistable} if it admits a proper regular model over~$\C[[t]]$ whose special fiber is reduced with local normal crossings (we do not require
simple normal crossings).
Let $\Delta \subset \C$ denote the unit disk and $\Delta^*=\Delta\setminus\{0\}$.
A \emph{degeneration} is a proper, flat, holomorphic map $\pi : \sX \to \Delta$, where~$\sX$ is a complex analytic manifold,
such that the complex analytic space $\pi^{-1}(t)$ is a connected manifold for every $t\neq 0$.
The degeneration~$\pi$ is \emph{semistable} if $\pi^{-1}(0)$ is a reduced divisor with local normal crossings.
A~\emph{model} of $\pi:\sX\to\Delta$ is a degeneration $\pi':\sX'\to\Delta$
such that the complex analytic manifolds $\pi^{-1}(\Delta^*)$ and $\pi'^{-1}(\Delta^*)$ are isomorphic over~$\Delta^*$.
A \emph{degeneration of~$K3$ surfaces} is a degeneration whose fibers above~$\Delta^*$ are~$K3$ surfaces.
It is \emph{Kulikov} if it is semistable and~$\sX$ has trivial canonical bundle.

\begin{thm}[Kulikov~\cite{kulikov}, Persson--Pinkham~\cite{perssonpinkham}]
\label{th:kulikov}
Let $\pi:\sX \to \Delta$ be a semistable degeneration of~$K3$ surfaces.
If the irreducible components of $\sX_0=\pi^{-1}(0)$ are algebraic, then~$\pi$ admits a Kulikov model.
\end{thm}

Note that the hypothesis ``the irreducible components of $\sX_0$ are algebraic'' is preserved when passing from a semistable model to another one
(see \cite[Lemma~3.1.1]{persson}).

Let $(A_i)_{i\in I}$ denote the family of irreducible components of~$\sX_0$.
Let $D \subset \sX_0$ denote the singular locus of~$\sX_0$.
Let $\Gamma_0=I$, let $\Gamma_1$ denote the set of irreducible components of~$D$,
and let $\Gamma_2$ be the set of triple points, \emph{i.e.}, the singular locus of~$D$.
The three sets $\Gamma_0$, $\Gamma_1$, $\Gamma_2$ may be organized into a semi-simplicial set~$\Gamma$,
in such a way that $x$ is a face of $y$ if and only if $y \subset x$.
Such a semi-simplicial set~$\Gamma$ is often referred to as \emph{the dual graph of~$\sX_0$} in the literature.
We write $|\Gamma|$ for the geometric realization of~$\Gamma$.

For $i \in I$, let $A_i'$ denote the normalization of~$A_i$, let $\nu_i:A_i' \to \sX_0$ denote the canonical map,
and set $D_i = \nu_i^{-1}(D)$.

An \emph{anticanonical pair} is a pair $(S,C)$ consisting of a smooth proper surface~$S$ and
a reduced, connected, nodal, singular curve~$C$ on~$S$, such that $\omega_S \simeq \sO_S(-C)$.
The curve~$C$ is either an irreducible nodal rational curve with a unique singular point, or a union of smooth rational curves forming a polygon.
If $\pi:\sX \to \Delta$ is a Kulikov degeneration of~$K3$ surfaces such that the irreducible components of $\sX_0$ are algebraic, then one of the following three possibilities must occur:
\begin{enumerate}
\item[(I)] $\sX_0$ is a smooth~$K3$ surface;
\smallskip
\item[(II)] $\sX_0$ is a chain of elliptic ruled surfaces, meeting along disjoint sections, with rational surfaces on either end;
\smallskip
\item[(III)] $\sX_0$ is a union of rational surfaces, $(A'_i,D_i)$ is an anticanonical pair for every $i\in I$,
and~$|\Gamma|$ is homeomorphic to the ($2$\nobreakdash-dimensional) sphere
\end{enumerate}
(see \cite[p.~26--29]{sagsoverview},
where on page~27, line~5, one should twist the \v{C}ech complex with orientation line bundles
in case~$\sX_0$ is not a simple normal crossings divisor).
We say that~$\pi$ is a degeneration of type~I, II, or~III accordingly.
Finally, a type~III degeneration of~$K3$ surfaces is \emph{in minus-one-form} if for every $i\in I$, every smooth (resp. singular) irreducible component of~$D_i$
has self-intersection~$-1$ (resp.~$1$) on~$A'_i$.

\begin{thm}[Miranda--Morrison~\cite{mirandamorrison}]
\label{th:mirandamorrison}
Any type~III degeneration of~$K3$ surfaces admits a type~III model which is in minus-one-form.
\end{thm}

To some extent, Theorem~\ref{th:mirandamorrison} reduces the study of~$K3$ surfaces over~$\C((t))$ with semistable reduction of type~III to that of non-negatively curved triangulations of the sphere.
This, in turn, quickly leads to unsolved problems in combinatorics
(see \cite[p.~22]{sagsoverview}, \cite{laza}).
It turns out that for the proof of Theorem~\ref{th:semiK3},
one can get by with a simple global combinatorial input, namely the well-known fact that the sphere cannot be tiled with hexagons only.

\begin{proof}[Proof of Theorem~\ref{th:semiK3}]
Given a degeneration
$\pi:\sX \to \Delta$, we shall
consider the complex of singular cohomology groups
\begin{align}
\label{eq:complexsepclosedanalytic}
\xymatrix{
H^1(X,\Q/\Z(1)) \ar[r] & H^2_A(\sX,\Q/\Z(1)) \ar[r] & \displaystyle\bigoplus_{i \in I} \frac{H^2(A_i,\Q/\Z(1))}{\CH_1(A_i)^\perp}
}
\end{align}
modelled on~\eqref{eq:complexsepclosed},
where $X = \sX \setminus \sX_0$ and $A=\sX_0$,
and where for any compact complex analytic manifold~$V$,
we denote by $\CH_1(V)^\perp$ the subgroup of $H^2(V,\Q/\Z(1))$ consisting of those classes whose restriction to $H^2(Z,\Q/\Z(1))$ vanishes for
all irreducible closed analytic subsets $Z \subset V$ of dimension~$1$.

\begin{prop}
\label{prop:checkkulikov}
Let $\pi:\sX \to \Delta$ be a Kulikov degeneration of~$K3$ surfaces. Assume the irreducible components of~$\sX_0$ are algebraic.
Assume, moreover, that if~$\pi$ is a degeneration of type~III, then it is in minus-one-form.
Then the complex~\eqref{eq:complexsepclosedanalytic} is exact.
\end{prop}

Let us postpone the proof of Proposition~\ref{prop:checkkulikov}, and instead deduce Theorem~\ref{th:semiK3} from it.
To this end, we first remark that the homology group of~\eqref{eq:complexsepclosedanalytic} is invariant under blow-ups with smooth centers.

\begin{lem}
\label{lem:blowupinvariance}
Let $\pi:\sX \to \Delta$ be a degeneration.
Let $\tilde\pi:\sXtilde \to \Delta$ be the degeneration obtained by blowing up~$\sX$ along a smooth complex analytic closed subspace of~$\sX_0$.
The homology groups of the complexes~\eqref{eq:complexsepclosedanalytic} associated with~$\pi$ and with~$\tilde\pi$ are isomorphic.
\end{lem}

\begin{proof}
Since any irreducible closed analytic subset of~$A$ of dimension~$1$ is contained in~$A_i$ for some~$i$,
the natural map $H^2(A,\Q/\Z(1))/\CH_1(A)^\perp \to \bigoplus_{i \in I} H^2(A_i,\Q/\Z(1))/\CH_1(A_i)^\perp$ is injective.
As a consequence, the complex~\eqref{eq:complexsepclosedanalytic} has the same homology group as the complex
\begin{align}
\label{eq:complexsepclosedanalyticbis}
\xymatrix{
H^1(X,\Q/\Z(1)) \ar[r] & H^2_A(\sX,\Q/\Z(1)) \ar[r] & \displaystyle\frac{H^2(A,\Q/\Z(1))}{\CH_1(A)^\perp}\rlap{\text{.}}
}
\end{align}
There is a natural morphism from the complex~\eqref{eq:complexsepclosedanalyticbis} to the analogous complex for the degeneration~$\tilde\pi$.
A diagram chase now implies the lemma, thanks to the following remarks.
First, as~$\sX$ is a manifold, the exceptional divisor $E \subset \sXtilde_0$ is a projective bundle over the center of the blow-up $L \subset \sX_0$.
Therefore, for any irreducible closed analytic subset $Z \subset A$ of dimension~$1$, there exists an irreducible closed analytic subset $\tilde Z \subset \tilde A$
of dimension~$1$ which maps birationally to~$Z$.  It follows that the pull-back map
\begin{align}
\label{eq:injh2ch1}
\frac{H^2(A,\Q/\Z(1))}{\CH_1(A)^\perp} \longrightarrow \frac{H^2(\tilde A,\Q/\Z(1))}{\CH_1(\tilde A)^\perp}
\end{align}
is injective.  Secondly, the first Chern class of~$\sO_{\sXtilde}(E)$ and the pull-back map induce a canonical decomposition
\begin{align}
H^2_{\tilde A}(\sXtilde,\Q/\Z(1)) =
H^2_A(\sX,\Q/\Z(1)) \oplus \Q/\Z\rlap{\text{,}}
\end{align}
and
if $\tilde Z \subset E$ denotes a line contained in a fiber of the projective bundle $E \to L$,
then $\deg\left(\sO_{\sXtilde}(E)|_{\tilde Z}\right)=-1$.
\end{proof}

Using Lemma~\ref{lem:blowupinvariance}, we now check that Proposition~\ref{prop:checkkulikov}
implies the validity of Theorem~\ref{th:semiK3} for~$K3$ surfaces defined over the function field of a complex curve.

\begin{lem}
\label{lem:k3biratinv}
Assume Proposition~\ref{prop:checkkulikov} holds.
Let $f:Y \to B$ be a morphism between smooth complex proper algebraic varieties, where~$B$ is a curve and the generic fiber of~$f$ is a~$K3$ surface.
Let $b \in B$.
Let $K \simeq \C((t))$ denote the completion of~$\C(B)$ at~$b$,
and let $Y_{\hat\eta} = Y \times_B \Spec(K)$.
If the fiber~$Y_b$ is a reduced divisor with local normal crossings,
then the group $A_0(Y_{\hat\eta})$ is divisible.
\end{lem}

\begin{proof}
Let $\Delta \subset B$ be a small disk around~$b$.
Denote by $\pi:\sX \to \Delta$ the restriction of~$f$ to $f^{-1}(\Delta)$.
By Theorem~\ref{th:kulikov} and Theorem~\ref{th:mirandamorrison}, there exists a
model $\pi':\sX'\to \Delta$ of~$\pi$ satisfying the following properties: $\pi'$ is Kulikov,
the irreducible components of~$\sX'_0$ are algebraic, and if~$\pi'$ is of type~III, then it is in minus-one-form.
Accordingly, by Proposition~\ref{prop:checkkulikov}, the complex~\eqref{eq:complexsepclosedanalytic} associated with~$\pi'$ is exact.
Let $Y'$ denote the complex analytic manifold obtained by gluing~$\sX'$ with $Y \setminus Y_b$ along $\sX' \setminus \sX'_0 = \sX \setminus \sX_0$.
By the weak factorization theorem for bimeromorphic maps between compact complex analytic manifolds \cite[Theorem~0.3.1]{akmw} applied to~$Y$ and~$Y'$,
and by Lemma~\ref{lem:blowupinvariance}, the exactness of the complex~\eqref{eq:complexsepclosedanalytic} associated with~$\pi'$
implies the exactness of the corresponding complex for~$\pi$.
The latter coincides, term by term, with the complex~\eqref{eq:complexsepclosed}
associated with the regular model $Y \times_B \Spec(\widehat{\sO_{B,b}})$ of $Y_{\hat\eta}$.
Applying Theorem~\ref{th:criterionsepclosed} now concludes the proof
(see Remark~\ref{rk:complexsepclosed}~(iii)).
\end{proof}

We finally deduce Theorem~\ref{th:semiK3} from Proposition~\ref{prop:checkkulikov} in full generality.
Let~$X$ be a semistable $K3$ surface over $\C((t))$.
Let $\sX$ be a proper regular semistable model of~$X$ over~$\C[[t]]$.
By a standard spreading out argument (see \cite[Proposition~5.1.2]{mustatanicaise}),
one can find a smooth complex algebraic curve~$B$, a smooth complex algebraic variety~$Y$,
a point $b \in B$ and a proper flat morphism $f:Y \to B$ such that the
special fiber of~$\sX$ is isomorphic to~$Y_b$
and such that the generic fiber of~$f$ is a~$K3$ surface.
After compactifying and resolving singularities, one may assume that~$B$ is proper.
Let us fix a $\C$\nobreakdash-linear isomorphism between the completion of~$\sO_{B,b}$ and $\C[[t]]$.
Let $Y_{\hat\eta}=Y \times_B \Spec(\C((t)))$.
By Lemma~\ref{lem:k3biratinv}, the group~$A_0(Y_{\hat\eta})$ is divisible.
It follows, by Remark~\ref{rk:complexsepclosed}~(iv), that $A_0(X)$ is divisible as well, thus completing the proof.

It remains to establish Proposition~\ref{prop:checkkulikov}.  Let $\pi:\sX \to \Delta$ be as in its statement.
According to Remark~\ref{rk:complexsepclosed}~(ii) (which also applies in this context), the exactness of~\eqref{eq:complexsepclosedanalytic} is equivalent to
that of the complex
\begin{align}
\xymatrix@C=2.5em{
\Q/\Z \ar[r] & (\Q/\Z)^I \ar[r]^(.355)\delta & \displaystyle\bigoplus_{i \in I} \frac{H^2(A_i,\Q/\Z(1))}{\CH_1(A_i)^\perp}\rlap{\text{,}}
}
\end{align}
where the first map is the diagonal inclusion and where~$\delta$ is as described in Remark~\ref{rk:complexsepclosed}.
We shall denote by $\delta_i : (\Q/\Z)^I \to
H^2(A_i,\Q/\Z(1))/\CH_1(A_i)^\perp$
the composition of~$\delta$ with the projection onto the $i$th summand.

Let us fix $\lambda = (\lambda_i)_{i \in I} \in (\Q/\Z)^I$ such that $\delta(\lambda)=0$, and show that all $\lambda_i$'s are equal.
If~$A$ is irreducible, there is nothing to prove.  Otherwise~$\pi$ is of type~II or of type~III.

Suppose~$\pi$ is a degeneration of type~II,
and number the irreducible components of the special fiber $I=\{0,\dots,n\}$ in such a way that $A_i \cap A_j \neq \emptyset$ if and only if $|i-j|\leq 1$.
Thus $A_0$ and $A_n$ are rational surfaces, and~$A_i$, for $i \in \{1,\dots,n-1\}$, is an elliptic ruled surface, \emph{i.e.}, a smooth surface ruled over an elliptic curve.

Let $C_i = A_i \cap A_{i+1}$ for $i \in \{0,\dots,n-1\}$.
As the canonical bundle of~$\sX$ is trivial, the adjunction formula implies that $C_{i-1}+C_i$ is an anticanonical divisor on~$A_i$
for any $i \in \{1,\dots,n-1\}$,
and that~$C_0$ (resp.~$C_{n-1}$) is an anticanonical divisor on~$A_0$ (resp.~on~$A_n$).
On the other hand, the self-intersection number~$(K_{A_i}^2)$ vanishes,
by \cite[Chapter~V, Corollary~2.11]{hartshorne}, and the curves~$C_{i-1}$ and~$C_i$ are disjoint.  Therefore $(C_{i-1}^2)_{A_i}+(C_i^2)_{A_i}=0$
for any $i \in \{1,\dots,n-1\}$.
(The subscript indicates that the intersection number is computed on~$A_i$.)
Now, by the triple point formula \cite[Corollary~2.4.2]{persson},
we have $(C_{i-1}^2)_{A_i}+(C_{i-1}^2)_{A_{i-1}}=0$ for all $i \in \{1,\dots,n\}$;
hence, in the end, we see that $(C_i^2)_{A_i}$ does not depend on~$i$, and in particular
\begin{align}
\label{eq:ka0kan}
(K_{A_0}^2)_{A_0}=(C_0^2)_{A_0}=(C_{n-1}^2)_{A_{n-1}}=-(C_{n-1}^2)_{A_n}=-(K_{A_n}^2)_{A_n}\rlap{\text{.}}
\end{align}
The self-intersection number of the canonical divisor of a (smooth, proper) minimal rational surface being either~$8$ or~$9$,
we deduce from~\eqref{eq:ka0kan} that the two surfaces~$A_0$ and~$A_n$ cannot both be minimal.

After possibly renumbering the~$A_i$'s, we may assume that~$A_0$ is not minimal.
Let $E \subset A_0$ be an exceptional curve.
The hypothesis that $\delta_0(\lambda)=0$ implies that $(C_0 \cdot E)_{A_0} \otimes (\lambda_1-\lambda_0)$ vanishes as an element of $\Q/\Z$.
On the other hand, since~$C_0$ is an anticanonical curve on~$A_0$ and~$E$ is exceptional, we have $(C_0\cdot E)_{A_0}=1$.  Hence $\lambda_0=\lambda_1$.
For $i \in \{1,\dots,n-1\}$, the hypothesis that $\delta_i(\lambda)=0$ amounts to the equality
\begin{align}
(C_{i-1}\cdot F)_{A_i} \otimes (\lambda_{i-1} - \lambda_i) + (C_i \cdot F)_{A_i} \otimes (\lambda_{i+1}-\lambda_i) = 0 \in \Q/\Z
\end{align}
for any divisor~$F$ on~$A_i$.
Letting~$F$ be a fiber of the ruling, we conclude that if~$\lambda_i=\lambda_{i-1}$, then $\lambda_{i+1}=\lambda_i$.
Thus, by induction, all $\lambda_i$'s are equal.

Suppose now~$\pi$ is a degeneration of type~III, in minus-one-form.
For $i\in I$, the surface~$A_i$ is rational and $(A_i',D_i)$ forms an anticanonical pair.
Let~$\gamma_i$ denote the set of irreducible components of~$D_i$.

\begin{defn}
Let $i \in I$.
We say that $j \in I$ is
a \emph{neighbour of~$i$} if $i\neq j$ and $A_i \cap A_j \neq\emptyset$.
We say that $C,C'\in \gamma_i$ are \emph{adjacent} if $C\neq C'$ and $C \cap C' \neq \emptyset$.
We say that~$i$ is \emph{consonant} if $\lambda_i=\lambda_j$ for every neighbour~$j$ of~$i$.
\end{defn}

For $i \in I$ and $C \in \gamma_i$,
we let $\mu_C=\lambda_j-\lambda_i$
if there exists a (necessarily unique) $j \in I \setminus \{i\}$ such that $\nu_i(C) \subset A_i \cap A_j$. Otherwise we let $\mu_C=0$.
Note that the hypothesis $\delta_i(\lambda)=0$ implies
\begin{align}
\label{eq:hypIII}
\sum_{C \in \gamma_i} (C\cdot F)_{A_i'} \otimes \mu_C = 0 \in \Q/\Z
\end{align}
for any divisor~$F$ on~$A_i'$.

\begin{lem}
\label{lem:consonantcriterion}
Let $i \in I$.  If there exist adjacent $C,C' \in \gamma_i$ such that $\mu_C=\mu_{C'}=0$, then~$i$ is consonant.
\end{lem}

\begin{proof}
As the curve~$D_i$ is a (necessarily reducible) polygon,
we may write $\gamma_i=\{C_1,\dots,C_n\}$ for some $n \geq 2$, in such a way that $C_i$ and~$C_j$ are adjacent if and only if $j-i=\pm 1 \mathrm{\ mod\ } n$.
Assume $\mu_{C_1}=\mu_{C_2}=0$.
Applying~\eqref{eq:hypIII} successively with $F=C_j$ for all $j \in \{2,\dots,n-1\}$,
we find that $\mu_{C_j}=0$ for all $j \in \{1,\dots,n\}$.  In other words~$i$ is consonant.
\end{proof}

Our goal is to prove that all $i \in I$ are consonant.
Lemma~\ref{lem:K3propagates} below reduces this to showing the existence of at least one consonant $i \in I$.

\begin{lem}
\label{lem:K3propagates}
Let $i \in I$.  If~$i$ is consonant, then so is every neighbour of~$i$.
\end{lem}

\begin{proof}
Let $j\in I$ be a neighbour of~$i$.
Choose a $C \in \gamma_j$ such that $\nu_j(C) \subset A_i \cap A_j$.
If~$i$ is the only neighbour of~$j$, then~$j$ is consonant.  Otherwise, we may further choose
a $C' \in \gamma_j$ adjacent to~$C$.
As~$i$ is consonant, we have $\mu_C=0$. In order to show that~$j$ is consonant,
it suffices to check that $\mu_{C'}=0$,
by Lemma~\ref{lem:consonantcriterion}.
Now if $\nu_j(C')$ is contained in the singular locus of~$A_j$, then $\mu_{C'}=0$ by definition.
Otherwise, there exists a unique neighbour~$k$ of~$j$ such that $\nu_j(C') \subset A_j \cap A_k$.
We then have $\nu_j(C \cap C') \subset A_i \cap A_k$, so that either $k=i$, or~$k$ is a neighbour of~$i$.
As~$i$ is consonant, in both cases we obtain $\lambda_k=\lambda_i$.  Hence $\lambda_j=\lambda_k$, or in other words $\mu_{C'}=0$,
which completes the proof.
\end{proof}

The following lemma summarizes well-known facts on anticanonical pairs which we shall need to
establish the existence of a consonant $i\in I$.

\begin{lem}
\label{lem:anticanonical}
Let $(S,C)$ be an anticanonical pair.
Let $C_1,\dots,C_n$ denote the irreducible components of~$C$.
Assume~$S$ is a rational surface.
If $n \geq 2$, assume $(C_i^2)=-1$ for all~$i$; if $n=1$, assume
$(C_1^2)=1$.
Then $n \leq 6$.  If moreover $n<6$, then for every $i\in\{1,\dots,n\}$, there exists an exceptional curve $E \subset S$ such that $(C_i\cdot E)=1$ and $(C_j\cdot E)=0$ for all~$j\neq i$.
\end{lem}

\begin{proof}
After renumbering the~$C_i$'s we may assume that $C_i \cap C_{i+1}\neq\emptyset$ for all~$i$.
Suppose that $n>6$ and let $F=C_1+2C_2+2C_3-C_4-2C_5-2C_6$. Then $(D^2)>0$, $(F^2)>0$ and $(D\cdot F)=0$, which contradicts the Hodge index theorem.
The second assertion relies on the classification of rational surfaces endowed with an anticanonical cycle of length at most~$5$
\cite[Theorem~(1.1)]{looijenga}, see \cite[Lemma~(11.5)]{mirandamorrison}, \cite[p.~107--108]{friedman}.
\end{proof}

Let~$n_i$ denote the number of irreducible components of~$D_i$.
Since~$|\Gamma|$ is homeomorphic to the sphere, Euler's formula may be written as
\begin{align}
\label{eq:euler}
\sum_{i \in I} \big(6 - n_i\big)=12
\end{align}
(see \cite[p.~21]{sagsoverview}).
On the other hand, as~$\sX_0$ is in minus-one-form,
Lemma~\ref{lem:anticanonical} implies that~$n_i \leq 6$ for all~$i$.
We conclude that $n_i<6$ for at least one~$i$. Let us fix such an $i\in I$.
For each $C \in \gamma_i$, there exists,
according to Lemma~\ref{lem:anticanonical},
an exceptional curve $E \subset A'_i$ such that $(C\cdot E)=1$ and $(C'\cdot E)=0$ for
all $C'\in\gamma_i\setminus\{C\}$.
Applying~\eqref{eq:hypIII} with $F=E$, we deduce that $\mu_C=0$.
Thus $\mu_C=0$ for all $C \in \gamma_i$, which means that~$i$ is consonant.
By Lemma~\ref{lem:K3propagates}, it follows that every element of~$I$ is consonant, in other words all $\lambda_i$'s are equal,
and the proofs of Proposition~\ref{prop:checkkulikov} and of Theorem~\ref{th:semiK3} are complete.
\end{proof}

\section{Homologically trivial zero-cycles need not be divisible}
\label{sec:counterexample}

We finally provide a counterexample to the injectivity of the cycle class
map $$\CH_0(X) \otimeshat \Zl \to H^{2d}(X, \Zl(d))$$
when~$X$ is a smooth projective surface over~$\C((t))$ or over the maximal unramified extension of a $p$\nobreakdash-adic field.
Thus, Theorem~\ref{th:H2alg-ksepclosed} does not extend to surfaces with a possibly nontrivial transcendental quotient of $H^2(\bar X,\Ql(1))$,
despite the case of semistable~$K3$ surfaces
treated in \textsection\ref{sec:semistableK3} and despite the fact that over $p$\nobreakdash-adic fields, Theorem~\ref{th:kfinite} applies to surfaces of arbitrary geometric genus.

\begin{thm}
\label{th:counterexample}
There exists a simply connected smooth projective surface~$X$ over $\C((t))$, with semistable reduction, such that $A_0(X)/2A_0(X)=\Z/2\Z$.

Similarly, for infinitely many prime numbers~$p$, there exists a
simply connected smooth projective surface~$X$ over the maximal unramified extension of a $p$\nobreakdash-adic field, with semistable reduction,
such that $A_0(X)/2A_0(X)=\Z/2\Z$.
\end{thm}

For any~$X$ as in the statement of Theorem~\ref{th:counterexample}, the kernels of the cycle class maps $\CH_0(X) \otimeshat \Z_2 \to H^4(X,\Z_2(2))$
and $\CH_0(X)/2\CH_0(X) \to H^4(X,\Z/2\Z)$
both have order~$2$
(see Remark~\ref{rk:complexsepclosed}~(iii)).

The surfaces constructed in Theorem~\ref{th:counterexample} have Kodaira dimension~$1$
and geometric genus~$3$.  They degenerate
into the union of two simply connected
surfaces of Kodaira dimension~$1$ and geometric genus~$1$, which meet transversally along an elliptic curve, with a $2$\nobreakdash-torsion normal bundle.
The nonzero element of $A_0(X)/2A_0(X)$ is represented by the difference of any two rational points of~$X$ which
specialize to distinct irreducible components of the special fiber.

\subsection{An elliptic surface}

We start by constructing the irreducible components of the desired special fiber.

\begin{prop}
\label{prop:existsalphabeta}
Let~$k$ be an algebraically closed field of characteristic~$0$.
There exists a simply connected smooth projective surface~$V$ over~$k$, endowed with a pencil $f:V \to \P^1_k$ of curves of genus~$1$, such that
\begin{enumerate}
\item the fiber $f^{-1}(0)$ has multiplicity~$2$ and its underlying reduced scheme is smooth;
\item all other fibers have multiplicity~$1$;
\item any divisor on~$V$ has degree divisible by~$4$ on the fibers of~$f$;
\item the total space of the Jacobian fibration of~$f$ is a~$K3$ surface.
\end{enumerate}
\end{prop}

By ``pencil of curves of genus~$1$'' we mean a morphism whose generic fiber is a geometrically irreducible curve of genus~$1$.
Before proving Proposition~\ref{prop:existsalphabeta}, let us recall a sufficient condition for
an elliptic surface to be simply connected.

\newcommand{\citeGS}{\cite[\textsection2, Theorem~1]{gurjarshastri}}
\begin{lem}[\citeGS]
\label{lem:simplyconnected}
Let~$V$ be a smooth projective surface
endowed with a pencil $f:V \to \P^1_k$ of curves of genus~$1$,
over an algebraically closed field~$k$ of characteristic~$0$.
Assume that~$f$ has at most one multiple fiber, and
at least one fiber whose underlying reduced scheme is not smooth.
Then~$V$ is simply connected.
\end{lem}

\begin{proof}[Proof of Proposition~\ref{prop:existsalphabeta}]
We shall construct~$V$ using Ogg--Shafarevich theory.  Let~$E$ be a~$K3$ surface over~$k$ with a pencil of elliptic curves $p:E\to \P^1_k$
(a section of~$p$ is understood to be chosen).
After a change of coordinates, we may assume that $p^{-1}(0)$ is smooth.
Let~$E_\eta$ denote the generic fiber of~$p$.
The smooth locus~$\sE$ of~$p$
identifies with the N\'eron model of~$E_\eta$
and thus
 has a natural group scheme structure over~$\P^1_k$.
Let $K=k(\P^1)$.
It is well known that the Leray spectral sequence for the inclusion of the generic point of~$\P^1_k$ gives rise to an exact sequence of \'etale cohomology groups
\begin{align}
\label{eq:oggshaf}
\xymatrix{
0 \ar[r] & H^1(\P^1_k,\sE) \ar[r] & H^1(K,E_\eta) \ar[r] & \displaystyle\bigoplus_{m \in \vphantom{\P^1_k}\smash{{\P^1_k}^{(1)}}} H^1(K_m,E_\eta) \ar[r] & 0\rlap{\text{,}}
}
\end{align}
where $\vphantom{\P^1_k}\smash{{\P^1_k}^{(1)}}$ denotes the set of codimension~$1$ points of~$\P^1_k$ and~$K_m$ stands for the completion of~$K$ at~$m$
(see \cite[Proposition~5.4.3 and Corollary~5.4.6]{cossecdolgacev}).
Moreover, there are isomorphisms
\begin{align}
\label{eq:h1br}
H^1(\P^1_k,\sE) = \Br(E) \simeq (\Q/\Z)^{b_2-\rho}\rlap{\text{,}}
\end{align}
where $b_2=22$ and~$\rho$ denotes the Picard number of~$E$
(see~\emph{op.\ cit.}, Theorem~5.4.3, noting that $H^3(E,\Zl)$ is torsion-free since~$E$ is a simply connected surface).

As~$E$ is a~$K3$ surface, we have $b_2-\rho>0$, so that $H^1(\P^1_k,\sE)$ contains
elements of any order, by~\eqref{eq:h1br}.
We fix an $\alpha \in H^1(\P^1_k,\sE)$ of order~$4$.

Recall that for any closed point $m \in \P^1_k$ of good reduction for~$E_\eta$, there is an isomorphism $H^1(K_m,E_\eta)\simeq(\Q/\Z)^2$
(see~\emph{op.\ cit.}, Theorem~5.4.1).
On the other hand, according to~\eqref{eq:h1br},
the group $H^1(\P^1_k,\sE)$ is divisible.
In view of~\eqref{eq:oggshaf},
it follows that there exists an element $\beta \in H^1(K,E_\eta)$ of order~$2$, whose image in
$H^1(K_0,E_\eta)$ has order~$2$, and whose image in
$H^1(K_m,E_\eta)$ vanishes for all closed points $m \in \P^1_k \setminus \{0\}$.

Let $\gamma=\alpha+\beta$.  We have now constructed a class $\gamma \in H^1(K,E_\eta)$ of order~$4$, whose image in
$H^1(K_0,E_\eta)$ has order~$2$, and whose image in $H^1(K_m,E_\eta)$ vanishes for all $m\neq 0$.
It~remains to be checked that the minimal proper regular model $f:V\to \P^1_k$ of the torsor~$V_\eta$ under~$E_\eta$ classified by~$\gamma$
satisfies the conclusion of Proposition~\ref{prop:existsalphabeta}.
Condition~(4) holds by construction, and~(1) follows from the fact that the image of~$\gamma$ in $H^1(K_0,E_\eta)$ has order~$2$
while $p^{-1}(0)$ is smooth (see~\emph{op.\ cit.}, Theorem~5.3.1 and Proposition~5.4.2).
Condition~(2) is equivalent to the vanishing of the image of~$\gamma$ in $H^1(K_m,E_\eta)$ for $m \neq 0$.
The genus~$1$ curve~$V_\eta$ has period~$4$, therefore its index is a multiple of~$4$ (see \cite[Proposition~5]{langtate};
it is in fact equal to~$4$, by \cite[Corollary~3]{ogg}), which implies~(3).
Finally we need to prove that~$V$ is simply connected.
As~$E$ is simply connected,
the morphism~$p$ cannot be smooth, and hence~$f$ has at least one fiber
whose underlying reduced scheme is not smooth, by \cite[Theorem~5.3.1]{cossecdolgacev}.
Thus Lemma~\ref{lem:simplyconnected} applies.
\end{proof}

\subsection{Persson's construction in mixed characteristic}
\label{subsec:persson}

Let~$V$ be the surface given by Proposition~\ref{prop:existsalphabeta}.
As~$V$ is simply connected and the class of a fiber of~$f$ in $\Pic(V)$ is divisible by~$2$,
we may consider, for a general $t \in \P^1(k) \setminus \{0\}$, the double cover of~$V$ branched along~$f^{-1}(t)$.
This is a smooth projective surface over~$k$.  Persson~\cite[Appendix~1, II]{persson} shows that when one lets~$t$ specialize to~$0$, this double cover degenerates,
with a regular total space, to the union of two copies of~$V$ glued along the elliptic curve $f^{-1}(0)_\red$.
In~order to construct the surface~$S$ of Theorem~\ref{th:counterexample} over the maximal unramified extension of a $p$\nobreakdash-adic field, we shall need
to adapt Persson's construction to the mixed characteristic setting.

Let~$t$ denote the coordinate of~$\P^1_\Z$, so that $\A^1_\Z=\Spec(\Z[t])$.
Assume we are given a discrete valuation ring~$R$ of characteristic~$0$ and residue
characteristic~$\neq 2$, a uniformizer~$\pi$ of~$R$, an irreducible regular
scheme~$\sV$ and a proper and flat morphism $\sV \to \P^1_R$
such that the divisor of the rational function~$g$ on~$\sV$ obtained by pulling back~$(t-\pi)/t$ may be written as
$D-2D'$ where~$D$ is a regular scheme and~$D'$ is smooth over~$R$.
We then consider $\sO_{\sV} \oplus \sO_{\sV}(-D')$
as an $\sO_{\sV}$\nobreakdash-algebra
with product $(a\oplus b)(c \oplus d)=(ac + gbd)\oplus (ad+bc)$,
and let~$\sX$ denote the corresponding finite flat $\sV$\nobreakdash-scheme of degree~$2$.
The following lemma summarizes the properties of~$\sX$.  Its proof is elementary and is left to the reader.

\begin{lem}
\label{lem:constrprop}
The scheme $\sX$ is irreducible, regular, and it is proper and flat over~$R$.
Letting~$k$ and~$K$ respectively denote the residue field and the quotient field of~$R$,
its fibers are described as follows.
\begin{enumerate}
\item The generic fiber $X = \sX \otimes_R K$ is a double cover of $V = \sV \otimes_R K$ branched only along the fiber of $V \to \P^1_K$ above $\pi\in\P^1(K)$.
It fits into a commutative square
\begin{align*}
\xymatrix@R=1.5em{
X \ar[d] \ar[r] & V \ar[d] \\
\P^1_K \ar[r] & \P^1_K
}
\end{align*}
in which the map $\P^1_K \to \P^1_K$ is a double cover branched along $\{0,\pi\}$.
The left-hand side vertical map is smooth above~$0$
and the square is cartesian above $\P^1_K \setminus \{0\}$.

\smallskip\addvspace{.2em}
\item The special fiber $A=\sX \otimes_R k$ is the scheme obtained by gluing two copies of $\sV \otimes_R k$ along the closed subscheme $D' \otimes_R k$
(see~\cite[\textsection1.1]{anantharaman}).
If $\sV \otimes_R k$ is smooth, then~$A$ has simple normal crossings.
\end{enumerate}
\end{lem}

\subsection{Proof of Theorem~\ref{th:counterexample}}

We are now in a position to prove the first statement of Theorem~\ref{th:counterexample}.
Let $R=\C[[t]]$ and $K=\C((t))$.  Let $f:V\to \P^1_\C$ be the surface given by Proposition~\ref{prop:existsalphabeta}
applied to $k=\C$.
Let~$\sX$ denote the $R$\nobreakdash-scheme associated in \textsection\ref{subsec:persson}
to $\sV = V \otimes_\C R$ and to the morphism $\sV \to \P^1_R$ deduced from~$f$ by base change.
According to Lemma~\ref{lem:constrprop}, the generic fiber $X=\sX\otimes_R K$ is a smooth projective surface over~$K$.
By Lemma~\ref{lem:simplyconnected} and
Lemma~\ref{lem:constrprop}~(1), it is simply connected.
As~$\sX$ is regular, the quotient of~$A_0(X)$ by its maximal divisible subgroup can be read off of $A=\sX \otimes_R k$, thanks to Theorem~\ref{th:criterionsepclosed}.
By Lemma~\ref{lem:constrprop}~(2), the variety~$A$ is reduced and has two irreducible components, which are both isomorphic to~$V$ and which meet
transversally along $f^{-1}(0)_\red$.
Moreover, the properties of~$V$ imply that the intersection number $(C\cdot f^{-1}(0)_\red)$ is even for any curve~$C$ lying on~$V$.
We are thus in the situation considered in Example~\ref{ex:firstex}~(iii); we conclude that $A_0(X)/2A_0(X)=\Z/2\Z$.

Let us turn to the second part of Theorem~\ref{th:counterexample}.
Proposition~\ref{prop:existsalphabeta} applied to $k=\bar \Q$ yields a simply connected smooth projective surface~$V$ defined over some number field $F\subset\bar\Q$, and a pencil $f:V \to \P^1_F$ of curves of genus~$1$
satisfying the properties (1)--(4) which appear in its statement.
We may extend~$V$ and~$f$ to a scheme~$\sV$ and a proper and flat morphism $f_\sO:\sV \to \P^1_\sO$
over the ring of $S$\nobreakdash-integers~$\sO$ of~$F$
for a large enough finite set~$S$
of places of~$F$,
and apply the construction of~\textsection\ref{subsec:persson} to the morphism deduced from $f_\sO$
by an extension of scalars from~$\sO$ to the localization of~$\sO$ at any maximal ideal.
This produces, for all but finitely many places~$v$ of~$F$, a smooth projective surface~$X$ over~$F$ and a regular proper model of $X \otimes_F F_v^\nr$ over~$\sO_v^\nr$
(where~$F_v^\nr$ denotes the maximal unramified extension of the completion of~$F$ at~$v$ and~$\sO_v^\nr$ is its ring of integers, with residue field $\bar \F_v$)
whose special fiber is the union of two copies of $\sV \otimes_\sO \bar \F_v$ glued along a smooth elliptic curve.
According to Lemma~\ref{lem:simplyconnected} and to Lemma~\ref{lem:constrprop}~(1), the surface~$X$ is simply connected.

In order to conclude as before by an application of Theorem~\ref{th:criterionsepclosed} and Example~\ref{ex:firstex}~(iii),
we must ensure that property~(3) of Proposition~\ref{prop:existsalphabeta} still holds for the pencil
\begin{align}
\label{eq:pencilmodv}
f_{\bar \F_v}:\sV \otimes_{\sO} \bar \F_v \to \P^1_{\bar \F_v}
\end{align}
obtained by reducing~$f_\sO$ modulo~$v$.
To this end we need to incorporate into the proof of Proposition~\ref{prop:existsalphabeta} some control over the reduction of the $2$\nobreakdash-torsion
classes in the Brauer group of the total space of the Jacobian fibration of~$f$.

Let $p:E\to \P^1_F$ denote the elliptic~$K3$ surface chosen at the beginning of the proof of Proposition~\ref{prop:existsalphabeta}.
If~$v$ is a place of~$F$ of good reduction for~$E$, we let $E_{\bar\F_v}$ denote the surface over $\bar \F_v$ obtained by reducing~$E$.

\begin{lem}
\label{lem:existsalpha}
There exists an $\alpha \in \Br(E \otimes_F \bar \Q)$ of order~$4$ whose image in $\Br(E_{\bar\F_v})$ has order~$4$ for infinitely many places~$v$ of~$F$.
\end{lem}

\begin{proof}
As the group $\Br(E \otimes_F \bar \Q)$ is divisible, it suffices to exhibit an $\alpha' \in \Br(E \otimes_F \bar \Q)$ of order~$2$ whose image in $\Br(E_{\bar\F_v})$ is nonzero
for infinitely many places~$v$.  If such an~$\alpha'$ did not exist, the specialization map $\tors{2}\Br(E \otimes_F \bar \Q) \to \tors{2}\Br(E_{\bar\F_v})$ between the corresponding $2$\nobreakdash-torsion subgroups
would have to vanish identically
for all but finitely many places~$v$,
since $\tors{2}\Br(E\otimes_F\bar \Q)$ is finite.  Now this map is surjective as it is a quotient of the specialization isomorphism
$H^2(E \otimes_F \bar \Q,\Z/2\Z) \isoto H^2(E_{\bar \F_v},\Z/2\Z)$.  Therefore we would have $\tors{2}\Br(E_{\bar\F_v})=0$ for all but finitely many places~$v$.
As~$E_{\bar\F_v}$ is a~$K3$ surface, this, in turn, would imply that the Picard number of $E_{\bar\F_v}$~is equal to~$22$
for all but finitely many places~$v$ of~$F$ (see~\eqref{eq:h1br}, which is valid in positive characteristic as well when restricted to the prime to~$p$ torsion).
But this is well known to be impossible (see \cite[Theorem~0.1]{bogomolovzarhin} for a stronger result).
\end{proof}

Recall that the surface $V \otimes_F \bar \Q$ was obtained, in the proof
of Proposition~\ref{prop:existsalphabeta}, as the minimal proper regular model, over~$\P^1_{\bar \Q}$, of the torsor under~$E_\eta$ classified
by a certain element $\gamma=\alpha+\beta$ in the corresponding Galois cohomology group, where~$\alpha$ was an arbitrary class of order~$4$ in~$\Br(E \otimes_F \bar \Q)$.
Let us now run the proof of Proposition~\ref{prop:existsalphabeta}
with the class~$\alpha$ given by Lemma~\ref{lem:existsalpha}
instead of an arbitrary~$\alpha$ of order~$4$.
As~$\beta$ has order~$2$, the conclusion of Lemma~\ref{lem:existsalpha} ensures that the reduction of~$\gamma$ modulo~$v$ has order~$4$ for infinitely many places~$v$ of~$F$.
In other words, for the surface~$V$ obtained by this procedure, the pencil~\eqref{eq:pencilmodv}
does satisfy property~(3) of Proposition~\ref{prop:existsalphabeta} for infinitely many places~$v$. Thus the proof of Theorem~\ref{th:counterexample} is complete.

\begin{rmk}
The degeneration $\sX\to \Spec(\C[[t]])$ constructed in the proof
of Theorem~\ref{th:counterexample} is a counterexample to \cite[Proposition~2.5.7]{persson}.
\end{rmk}

\subsection{A new example over a \texorpdfstring{$p$\nobreakdash-adic}{𝑝-adic} field}

Let~$X$ be the surface given by Theorem~\ref{th:counterexample} over the maximal unramified extension~$K$ of a $p$\nobreakdash-adic field,
for some~$p>2$.
There exist a $p$\nobreakdash-adic field~$K_0$ contained in~$K$, a surface~$X_0$ over~$K_0$,
and a zero-cycle~$z_0$ on~$X_0$,
such that $X=X_0 \otimes_{K_0}K$ and the image of~$z_0$
in $A_0(X)/2A_0(X)$ is nonzero.
We can moreover assume that $X_0(K_0)\neq\emptyset$,
since $X(K)\neq\emptyset$.
The cycle class of~$z_0$ belongs
to the kernel of the natural map $H^4(X_0,\Z/2\Z) \to H^4(X,\Z/2\Z)$.
After replacing~$K_0$ with a larger $p$\nobreakdash-adic field contained in~$K$,
we may assume that the cycle class of~$z_0$ itself vanishes,
in view of the fact that
$H^4(X,\Z/2\Z)=\varinjlim H^4(X_0 \otimes_{K_0}K_1,\Z/2\Z)$
where the direct limit ranges over all finite subextensions~$K_1/K_0$ of~$K/K_0$
(see \cite[Chapter~III, Lemma~1.16]{milne}).

We thus obtain an example of a smooth projective surface~$X_0$ defined over a $p$\nobreakdash-adic field~$K_0$,
with~$X_0(K_0)\neq\emptyset$,
such that
the cycle class map
\begin{align}
\CH_0(X_0)/n\CH_0(X_0) \to H^4(X_0,\Z/n\Z(2))
\end{align}
fails to be injective for some integer~$n>1$.
Another example of such a surface was given by Parimala and Suresh~\cite[\textsection8]{parimalasuresh}.
Contrary to the example of \emph{loc.\ cit.},
the surface~$X_0$ is simply connected and the kernel of the cycle class map remains
nontrivial over the maximal unramified extension of~$K_0$.

\begin{rmk}
\label{rk:nonuniruledai}
On the other hand, the defect of injectivity of the cycle class map with \emph{integral} coefficients does not descend from~$X$ to~$X_0$.
In fact, it follows from Theorem~\ref{th:kfinite} that the cycle class map $\CH_0(X_0) \otimeshat \Zl \to H^4(X_0,\Zl(2))$ is injective
for any $\ell \neq p$.  To verify this claim, recall that~$X_0$ admits a proper regular model whose special fiber has simple normal crossings
and is the union of two copies of a surface~$V_0$ carrying a pencil~$f_0$ of curves of genus~$1$; moreover,
the total space of the Jacobian fibration of~$f$ is a~$K3$ surface~$E_0$.
The Tate conjecture holds for~$E_0$ (see~\cite{artinsd}).
With the help of the correspondences between~$E_0$ and~$V_0$
defined in \cite[Proof of Prop.~4]{blochkaslieberman},
one deduces that it also holds for~$V_0$. Thus all of
the hypotheses of Theorem~\ref{th:kfinite} are satisfied.
\end{rmk}

\appendix\newpage
\section{The structure of algebraic \texorpdfstring{$1$}{1}-cycles for regular schemes over a strictly henselian discrete valuation ring\texorpdfstring{\\\medskip}{ (}by Spencer Bloch\texorpdfstring{}{)}}
\smallskip

\label{appendix}

This appendix gives a simpler proof (without use of blowups) of Theorem 1.16 in \cite{saitosato} which is used in the paper of H\'el\`ene Esnault and Olivier Wittenberg. I am indebted to them for suggestions and corrections and also for showing me drafts of their paper and providing me with references \cite{GLL} and \cite{saitosato}.

Let $T=\Spec(\Lambda)$, where $\Lambda$ is a strictly henselian excellent discrete valuation ring with quotient field $K$ of characteristic $0$ and separably closed residue field $k= \Lambda/\pi\Lambda$. Let $\sX$ be a regular scheme, flat and projective over $T$ with fibre dimension $d$. Let 
\eq{}{A:= \sX\times_T \Spec k
} 
be the closed fibre. 
Let $\CH_1(\sX)$ be the Chow group of algebraic $1$\nobreakdash-cycles on $\sX$.
Let $F:= \bigoplus_i \Z\cdot A_i$ be the free abelian group on the irreducible components~$A_i$ of~$A$. We have a map \cite[\textsection20.1]{fulton}
\eq{2}{\deg: \CH_1(\sX) \to F^\vee = \text{Hom}(F,\Z);\quad z \mapsto \{A_i \mapsto \deg(z\cdot A_i)\}.
}
Define
\eq{}{\CH_1(\sX)^0 := \ker(\deg) \subset \CH_1(\sX). 
}

\begin{thm}[\cite{saitosato}, Theorem 1.16] \label{app:thm1} Let notation be as above, and suppose given $n\ge 2$ with $1/n\in k$. Then the cycle map yields an isomorphism
\eq{}{\CH_1(\sX)/n\CH_1(\sX) \cong H^{2d}_{\text{\'et}}(\sX, \mu_n^{\otimes d}).
}
\end{thm}

\begin{lem}With notation as above, $\CH_1(\sX)^0$ is divisible prime to the characteristic of the residue field $k$. 
\end{lem}

\begin{proof}[Proof of theorem]
Let $p=\text{char}(k)$. Since $k$ is assumed to be separably closed, given $i$, we can find $x\in A_i$ a $k$-point not lying in any other $A_j$ (\cite[ 17.15.10(iii))]{ega44}. Since $\sX$ is regular, we can find an irreducible closed subscheme $\sV \subset \sX$ of dimension $1$ meeting $A_i$ transversally at $x$. Since $\Lambda$ is strictly henselian, $\sV$ is local, meeting $A$ only at $x$. We have $\sV\cdot A_i = 1$ and $\sV\cdot A_j=0,\ j\neq i$. It follows that $\deg: \CH_1(\sX)/n\CH_1(\sX) \surj \text{Hom}(F,\Z/n\Z)$, and assuming the lemma, this map is an isomorphism. On the other hand, by base change, writing $A=\bigcup_i A_i$ we have
\eq{6b}{H^{2d}_{\text{\'et}}(\sX, \mu_n^{\otimes d})\cong H^{2d}_{\text{\'et}}(A, \mu_n^{\otimes d}) \cong \text{Hom}(F, \Z/n\Z).
}
The assertion of the theorem follows.
\end{proof}

\begin{proof}[Proof of Lemma]
By an application of the Gabber--de Jong theorem on alterations \cite[Theorem~1.4]{Illusie}
and a trace argument, we may
assume that the reduced special fiber $A_{\rm red}$ is a simple normal crossings
divisor on $\sX$.

Suppose first the fibre dimension $d=1$ so $\sX$ is a curve over $T$. In this case, $\CH_1(\sX) = \text{Pic}(\sX)$ and the assertion of the lemma follows from the Kummer sequence
\eq{6}{\text{Pic}(\sX) \xrightarrow{n}  \text{Pic}(\sX) \to H^2(\sX, \mu_n)
}
together with \eqref{6b} above. 

Now take $d\ge 2$. Let $Z=\sum c_\mu Z_\mu$ be a $1$-cycle on $\sX$ such that the $Z_\mu$ are integral schemes and $\deg(Z\cdot A_i)=0$ for all $i$.
By an elementary moving lemma, we can assume none of the $Z_\mu$ are supported on $A$.  Indeed, assuming $Z_\mu \subset A$, we take intersections of hypersurface sections to construct an irreducible closed subscheme $\sY \subset \sX$ of dimension $2$ such that $Z_\mu \subset \sY\cap A$ and $\sY$ meets the irreducible components of $A$ transversally at the generic point of $Z_\mu$. In particular, $\sY$ is normal at the generic point of $Z_\mu$ so we can find a function $g$ in the function field of $\sY$ which has multiplicity $1$ along $Z_\mu$ and is a unit at every other component of $\sY\cap A$. The cycle $-(g)_\sY+Z_\mu$ is then rationally equivalent to $Z_\mu$ and has no component on $A$.  

A more serious moving lemma (\cite{GLL}, Theorem~2.3) permits us to assume that the $Z_\mu$ do not meet the
higher strata of the fibre $A$, i.e.\ for all $\mu$ and all $i\neq j$ we have $Z_\mu\cap A_i\cap A_j = \emptyset$.

Let $S_\mu \to Z_\mu$ be the normalization. Because $\Lambda$ is excellent,  $S_{\mu}$ is finite over $T$ \cite[Prop.~(7.8.6)]{ega42}.
Thus for $N\gg 0$ we can find a closed immersion $\xi: \coprod S_\mu \inj \P^N_T$ of schemes over $T$. In this way, we can build a diagram
\eq{7a}{\begin{CD} \coprod  S_\mu @>\iota >\inj > \sX\times_{T}\P^N_{T} \\
@VVV  @V pr_\sX VV \\
\coprod Z_\mu @>>> \sX \\
@VVV @VVV \\
T@= T. 
\end{CD}
}
The map $\iota$ is a closed immersion of regular schemes. Let $I \subset \sO_{\sX\times_T \P^N_T}$ be the corresponding ideal. Fix a very ample line bundle $\sO_{\sX\times_T \P^N_T}(1)$.
We denote by $\sO_{A_{{\rm red}} \times_k \P^N_k}(1)$ its restriction to $A_{{\rm red}} \times_k \P^N_k$.

We consider complete intersections 
$$\sC=\sC(\sigma_1,\dotsc,\sigma_{d+N-1}) \subset \sX\times_T \P^N_T$$ 
defined by sections $\sigma_i \in \Gamma(\sX\times_T \P^N_T, I\sO(M))$ for $M\gg 0$.  

\begin{lem} \label{app:lem_curve}
There are sections $ \sigma_1,\dotsc,\sigma_{d+N-1}$ such that $\sC$ is regular of relative dimension~$1$ over $T$, and meets all faces of  the normal crossings divisor $A_{\text{red}}\times_k \P^N_k$ transversally.
\end{lem}

\begin{proof}
By assumption, the $S_{\mu}$ are regular of dimension $1$ and they meet the $(d+N)$-dimensional reduced closed fibre $A_{\text{red}}\times_k \P^N_k$ at a finite set of regular points $q_{\mu,i}$. Because $T$ is henselian, there is one connected component $S_{\mu,i}$ for each closed point $q_{\mu,i}$. The residue fields of the $q_{\mu,i}$ may be inseparable over $k$. These intersections may not be transverse, but the tangent spaces
\eq{8b}{t_{S_{\mu,i}\cap (A_{\text{red}}\times_k \P^N_k),q_{\mu,i}}\subset t_{A_{\text{red}}\times_k \P^N_k,q_{\mu,i}}
}
have dimension $\le 1$.  Set  $\bar I ={\rm Im} ( I \to \sO_{A_{{\rm red}}\times_k \P^N_k}).$ 
 For $M\gg 0$ we can, by general position arguments, arrange the following conditions to hold:
 \begin{itemize}
 \item[(i)]  The restriction
$$\mkern60mu\Gamma(I\otimes_{\sO_{\sX\times_T \P^N_K}} \sO_{\sX\times_T \P^N_K}(M)) \to 
\Gamma(\bar I\otimes_{\sO_{A_{{\rm red}} \times_k \P^N_k}}  \sO_{A_{{\rm red}} \times_k \P^N_k}(M))$$
 is surjective.
\item[(ii)] There exist ${\sigma}_1,\dotsc, {\sigma}_{d+N-1}\in \Gamma(I\otimes_{\sO_{\sX\times_T \P^N_K}} \sO_{\sX\times_T \P^N_K}(M)) $
such that  the subscheme $\bar{\sC}\subset A_{{\rm red}}\times_k \P^N_k  $ defined by 
$(\bar{\sigma}_1, \ldots, \bar{\sigma}_{d+N-1}) $ meets all the faces of the normal crossings divisor $A_{\text{red}}\times_k \P^N_k$ transversally\footnote{The referee inquires about the proof of (ii). To simplify we change notation and take~$W$ to be a smooth, projective variety over an infinite field $k$. Let~$\sO_W(1)$ be an ample line bundle and $J\subset \sO_W$ be
an ideal such that $\sO_W/J$ is supported on a finite set $\{q_i\}$ of closed points. The assertion becomes that for $N\gg 0$ there exists a non-empty open set $U$ in the affine space of sections $\Gamma(W, J(N))$ such that for $\sigma\in U$ the zero set $\sV(\sigma) \subset W$ is smooth away from the support of $\sO_W/J$ and further, at any point $q\in \{q_i\}$ with maximal ideal $\mathfrak m \subset \sO_{W,q}$ such that $J_q/(J_q\cap \mathfrak m^2) \neq (0)$, we have that the image of $\sigma$ in $J_q/(J_q\cap \mathfrak m^2)$ is non-zero. 
Let $\pi: W' \to W$ be the blowup of $J$, and let $\sL$ be the tautological line bundle for $W'$. For $N\gg 0$, its twist $\sL\otimes \pi^*\sO_W(N)$ is very ample on $W'$ and moreover $\Gamma(W, J(N))\cong \Gamma(W', \sL\otimes \pi^*\sO_W(N))$. The classical Bertini theorem now says that for a non-empty Zariski-open set $U'$ of sections~$s$, the zero subscheme $\sV(s) \subset W'$ is smooth away from the exceptional divisor (which may be singular).  Clearly, for $N\gg 0$ there is another non-empty open $U''$ of sections with non-zero image in $J_q/(J_q\cap \mathfrak m^2)$ for all $q\in \{q_i\}$. Then we take $\sigma_1\in U'\cap U''\neq \emptyset$. Repeating this construction yields the necessary $\sigma_i$.}. Then one has an inclusion of tangent spaces
\[ t_{S_{\mu,i}\cap (A_{\text{red}}\times \P^N_k),q_{\mu,i}} \subset  t_{\overline{\sC},q_{\mu,i}}, \]
where the lefthand side is as in \eqref{8b}, and  $\dim t_{\overline{\sC},q_{\mu,i}} = 1$.
\item[(iii)] The intersections of $\overline{\sC}\cap (A_{i,\text{red}}\times \P^N_k)$ with the top dimensional strata are irreducible
(they are smooth by \cite{kleimanaltman}, Theorem~7
and connected by \cite{jouanolou}, Theorem~6.10 and \cite{ega43},  Proposition~15.5.9).
\end{itemize}
The subscheme  $\sC\subset \sX\times_T \P^N_T$ defined by 
$(\sigma_1, \ldots, \sigma_{d+N-1})$ then satisfies the conditions of Lemma~\ref{app:lem_curve}.
\end{proof}
Let $f: \sC \to \sX$ denote the projection $\sX\times_T \P^n_T \to \P^n_T$ restricted to $\sC$. 
By construction, the irreducible components $B_i$ of the special fibre $B\subset \sC$ are the pullbacks $B_i=f^*A_i$. It follows that the cycle $S:= \sum c_\mu S_\mu$ satisfies
\eq{}{S\cdot B_i = S\cdot f^*A_i = f_*S\cdot A_i = Z\cdot A_i=0.
}
We conclude from the case $d=1$ that the cycle $S$ on $\sC$ is divisible prime to $\text{char}(k)$, and hence $Z=f_*S$ is divisible as well.
\end{proof}

\begin{rmk}
\label{app:rmk}
The statement of Theorem~\ref{app:thm1} also holds when the residue
field~$k$ is finite.
This is \cite[Theorem 1.16]{saitosato} if $A_{\rm red}$ has normal crossings,
and the general case follows as above by an application of the Gabber--de
Jong theorem.
When~$k$ is finite, the arguments of Theorem~\ref{app:thm1} may be
used instead of Step~1 and Step~3 of {\em loc. cit.}, \textsection8, thus leading to a simpler proof
of \cite[Theorem 1.16]{saitosato}.
\end{rmk}

\newpage
\bibliographystyle{amsalpha}
\bibliography{zcl}
\end{document}